\documentclass[11pt,reqno]{amsproc}
\usepackage[margin=1in]{geometry}
\usepackage{amsmath, amsthm, amssymb}
\usepackage[colorlinks=true, pdfstartview=FitV, linkcolor=blue,citecolor=blue, urlcolor=blue]{hyperref}
\usepackage[abbrev,lite,nobysame]{amsrefs}
\usepackage{times}
\usepackage[usenames,dvipsnames]{color}
\usepackage{bbm}
\usepackage{mathtools}
\usepackage[normalem]{ulem}
\usepackage{subcaption}

\mathtoolsset{showonlyrefs=true}
\usepackage{dsfont}

\newcommand{\jap}[1]{\left\langle #1 \right\rangle}

\def\ff{{\widehat f}}

\def\eps{\varepsilon}
\def\e{{\rm e}}

\def\C{{\mathbb C}}

\def\dd{{\rm d}}
\def\ddt{{\frac{\dd}{\dd t}}}
\def\R {\mathbb{R}}

\def\ZZ {\mathbb{Z}}
\def\Re{{\rm Re}}
\def\Im{{\rm Im}}
\def \l {\langle}
\def \r {\rangle}
\def\T {{\mathbb T}}
\def\de{{\partial}}
\newcommand\uu {\boldsymbol{u}}
\newcommand\vv {\boldsymbol{v}}
\def\gr {\boldsymbol{g}}

\def\bb {\beta}
\def\QQ{Q}
\def\RR{R}

\definecolor{orange}{rgb}{1.0, 0.49, 0.0}
\definecolor{green}{rgb}{0.0, 0.5, 0.0}
\definecolor{brown}{rgb}{0.43, 0.21, 0.1}

\newcommand{\dt}{\partial_t}

\newcommand{\dy}{\partial_y}
\newcommand{\dX}{\partial_X}
\newcommand{\dY}{\partial_Y}
\newcommand{\dXX}{\partial_{XX}}

\newcommand{\norma}[2]{\left\lVert #1 \right\rVert_{#2}}
\newcommand{\norm}[1]{\left\lVert #1 \right\rVert}

\newtheorem{proposition}{Proposition}[section]
\newtheorem{theorem}[proposition]{Theorem}
\newtheorem{corollary}[proposition]{Corollary}
\newtheorem{lemma}[proposition]{Lemma}
\theoremstyle{definition}

\newtheorem{remark}[proposition]{Remark}

\numberwithin{equation}{section}

\title[Linear inviscid damping for stratified fluids]{Linear inviscid damping for shear flows near Couette \\
in the 2D stably stratified regime}

\author[R. Bianchini]{Roberta Bianchini}
\address{IAC, Consiglio Nazionale delle Ricerche, 00185 Rome, Italy}
\email{r.bianchini@iac.cnr.it}

\author[M. Coti Zelati]{Michele Coti Zelati}
\address{Department of Mathematics, Imperial College London, London, SW7 2AZ, UK}
\email{m.coti-zelati@imperial.ac.uk}

\author[M. Dolce]{Michele Dolce}
\address{GSSI - Gran Sasso Science Institute, 67100, L'Aquila, Italy}
\email{michele.dolce@gssi.it}


\begin{document}



\maketitle

\begin{abstract}
	We investigate the linear stability of shears near the Couette flow for a class of 2D incompressible stably stratified fluids. Our main result
	consists of nearly optimal decay rates for perturbations of stationary states whose velocities are monotone 
	shear flows $(U(y),0)$ and have an exponential density profile. In the case of the Couette flow $U(y)=y$, 
	we recover the rates predicted by Hartman in 1975, by adopting an explicit point-wise approach in frequency space. As a by-product, this implies
	optimal decay rates as well as Lyapunov instability in $L^2$ for the vorticity. For the previously unexplored 
	case of more general shear flows close to Couette, the inviscid damping results follow by a weighted energy estimate. Each outcome concerning the stably stratified regime applies to the Boussinesq equations as well.
	Remarkably, our results hold under the celebrated \emph{Miles-Howard criterion} for stratified fluids.
\end{abstract}

\tableofcontents

\section{Introduction}
This article is concerned with the linear asymptotic stability of shear flows close to Couette for a class of two-dimensional incompressible stably stratified fluids without viscosity. Stability of \emph{laminar equilibria}, to which the Couette flow with velocity vector $(y,0)$ belongs, is a widely investigated problem in the general context of hydrodynamic stability, a field where several major open questions still remain unanswered. We refer to \cite{Gallay, BGM18} for two highlighting reviews on the topic. The stability results provided here are explicitly quantified in terms of decay rates due to the phenomenon known as \emph{inviscid damping}, caused by vorticity mixing. A brief review on the existing literature of linear and nonlinear inviscid damping is presented below.

As pointed out later on, from the rigorous mathematical viewpoint, the majority of inviscid damping results belongs to the framework of homogeneous fluids, namely incompressible flows with constant density (and zero viscosity, see \cite{BGM18} for a discussion about that). On the other hand, the physics literature on the stability of non-homogeneous fluids started several decades ago and provided one of the first remarkable answers in the early 60's, with a criterion for spectral stability of stratified flows due to Miles \cite{Miles61} and Howard \cite{Howard61}. The mathematical picture on the stability of shear flows for variable density fluids is therefore far from being complete and only \emph{linear} stability results are available up to now
\cites{ADM20,Lin18} . 

A relevant class of fluids with variable density is represented by the non-homogeneous incompressible Euler equations 
\begin{equation}\label{eq:nonlin}
	\begin{aligned}
	 \rho(\de_t+\uu\cdot\nabla)\uu+\nabla  p&=-\rho \, \gr,\\
	(\de_t+\uu\cdot\nabla)\rho&=0,\\
	\nabla\cdot \uu&=0, 
	\end{aligned}
\end{equation}
where $\uu=(u^x, u^y) $ is the velocity field, $\rho$ is the density, $p$ is the pressure and  $\gr=(0, \mathfrak{g})$ is (minus) gravity. Models of incompressible fluids have a large range of application, in particular in oceanography, see \cite{Rieutord}, where both incompressibility and absence of viscosity are very good approximations of reality. 

The system under investigation in this paper belongs exactly to this framework. On the spatial domain $(x,y)\in \T\times \R$ (a periodic channel), we consider 
\begin{equation}\label{eq:EulerBomega}
	\begin{aligned}
	(\de_t + U(y)\de_x)(\omega -\beta\dy \psi ) -(U''(y)-\beta U'(y))\de_x\psi&=-\RR\de_xq,\\
	(\de_t+ U(y)\de_x)q&=\de_x\psi,\\
	\Delta\psi &=\omega,
	\end{aligned}
\end{equation}
endowed with the suitable initial data
\begin{align}
\omega(0,x,y)=\omega_{in}(x,y),\qquad q(0,x,y)=q_{in}(x,y).
\end{align}
The above linear system is written in terms of the unknowns $\omega, \psi$ and $q$, which represent perturbations 
of a vorticity, a streamfunction and a (scaled) density, respectively. In particular, the velocity vector 
$(v^x,v^y)=\nabla^\perp \psi=(-\de_y\psi,\de_x\psi)$ is recovered via the Biot-Savart law
\begin{align}\label{eq:biot}
v^x=-\de_y\Delta^{-1}\omega,\qquad v^y=\de_x\Delta^{-1}\omega.
\end{align}
Moreover, $R>0$ and $\beta \geq 0$ are \emph{a priori} two constant and independent parameters. 
As explained in the subsequent Sections \ref{sub-1.1}-\ref{sub-1.2}, the potential interdependence of the constants $R$ and $\beta$ establishes the distinction between two different physically relevant models: the case of shear flows in the Boussinesq approximation 
(for $\beta=0$ and $R>1/4$) and the case of an exponentially stratified fluid (for $\beta>0$ and $R=\beta\mathfrak{g} >1/4$).

\subsection{Linear inviscid damping and main results} 
Since the $x$-average is a conserved quantity for the dynamics of \eqref{eq:EulerBomega}, it is convenient to introduce the notation
\begin{align}
\l f\r_x=\frac{1}{2\pi}\int_{\mathbb{T}}f(x,y)\dd x
\end{align}
to denote the $x$-average of a function $f:\T\times\R\to\R$. We prove a quantitative asymptotic stability result when
the background shear is close to the Couette flow $U(y)=y$. The main result of this paper reads as follows.
\begin{theorem}\label{thm:CloseCouette}
	Let $R>1/4$ and $\beta\geq0$ be arbitrarily fixed.  There exist $C_0>1$ and a small constant $\eps_0=\eps_0(\beta,R)$ with the following property. If
	$\eps\in(0,\eps_0]$ and
	\begin{align}\label{eq:closetoC}
	\| U'-1\|_{H^6}+\| U''\|_{H^5}\leq \eps,
	\end{align}
	then for every $t\geq0$ we have
	\begin{align}
	\norm{q(t)-\l q\r_x}_{L^2}+\norm{v^x(t)-\l v^x\r_x}_{L^2}\lesssim&\frac{1}{\l t\r^{\frac12-\delta_\eps}}\left(\norm{\omega_{in}-\l\omega_{in}\r_x}_{H^\frac12}+\norm{q_{in}-\l q_{in}\r_x}_{H^{\frac32}}\right),\label{eq:main1}\\
	\norm{v^y(t)}_{L^2}\lesssim& \frac{1}{\langle t \rangle^{\frac32-\delta_\eps}}\left(\norm{\omega_{in}-\l\omega_{in}\r_x}_{H^{\frac32}}+\norm{q_{in}-\l q_{in}\r_x}_{H^{\frac52}}\right),\label{eq:main2}
	\end{align}
where $\delta_\eps=C_0\eps$.
\end{theorem}
Assumption \eqref{eq:closetoC} is a quantification of how close the background shear is to the linear profile given by the Couette flow $U(y)=y$, which
corresponds to $\eps=0$.
Theorem \ref{thm:CloseCouette} is a consequence of more general stability estimates in $H^s$, for any $s\geq0$ (see 
Theorems \ref{thm:Couette-main} and \ref{thm:Close-main} below). It is worth mentioning here that the constants hidden in the symbol 
$\lesssim$ in \eqref{eq:main1}-\eqref{eq:main2} blow up exponentially fast as $R\to1/4$.

The decay estimates \eqref{eq:main1}-\eqref{eq:main2} are an example of \emph{inviscid damping}. This terminology, introduced in 
analogy with Landau damping \cite{MV11}, refers to a stability mechanism due to vorticity mixing. Studied since
Kelvin \cite{Kelvin87} and Orr \cite{Orr07}, the physics literature is quite vast (see \cites{DrazinReid81,Yaglom12,SH01,Drazin02} 
and references therein). Of interest to the case
treated here, we  mention that the mixing phenomenon behind inviscid damping also appears in stratified shear flows \cites{Majda03,camassa_viotti_2013}.

In the case of the 2D Euler equations with constant density, the \emph{nonlinear} asymptotic stability of the Couette flow
was obtained in \cite{BM15}, and much more recently the cases of monotone shears 
\cites{IJ2018,MZ20,IJ20} and the point-vortex \cite{IJ2019} were treated. In the linearized setting, more general flows can be analyzed \cites{BCZV17,WZZ18,WZZ17,WZZkolmo17,WZZ15,JIA19,CZZ18,Zcirc17,Zillinger14,Zillinger15,Zillinger19,GNRS18}. 

However, in the density-dependent case, vorticity mixing is not the only relevant effect. This is clarified in the special 
case of the Couette flow $U(y)=y$, where a  proof which is \emph{point-wise} in frequency space can be obtained, 
and gives the following instability result.

\begin{corollary}
	\label{cor:Lyinst}
Let $U(y)=y$ in \eqref{eq:EulerBomega}, and assume that $R>1/4$ and $\beta\geq0$. Then we have the lower bound
\begin{equation}
\label{bd:Lyomega}
\norm{\omega(t) -\l \omega\r_x}_{L^2}+\norm{\nabla q(t)-\l \nabla q\r_x}_{L^2} \gtrsim\langle t \rangle^{\frac12}\left[\norm{\omega_{in}-\l \omega_{in}\r_x}_{L^{2}_xH^{-1}_y}+\norm{q_{in}-\l q_{in}\r_x}_{H^{1}_xL^{2}_y}\right]
\end{equation}
and the upper bound
\begin{equation}
\norm{\omega(t)-\l \omega\r_x}_{L^2}+\norm{\nabla q(t)-\l \nabla q\r_x}_{L^2}\lesssim
\langle t \rangle^{\frac12}\left[\norm{\omega_{in}-\l \omega_{in}\r_x}_{L^2}+\norm{q_{in}-\l q_{in}\r_x}_{H^1}\right]
\end{equation}
for every $t\geq0$.
\end{corollary}

Corollary \ref{cor:Lyinst} highlights the presence of a secular instability, since $\omega$ and $\nabla q$ \emph{may} grow as $t^{1/2}$.
Interestingly, if $q$ was only transported by the Couette flow, then $\nabla q$ may grow as $t$, while in this case the system dynamics produce
a possible slower growth (as clear also from the $L^2$ decay of $q$ of Theorem \ref{thm:CloseCouette}). This instability for the vorticity is also the reason why  the inviscid damping rates of Theorem \ref{thm:CloseCouette} are slower than the ones of the constant density case, and confirms what predicted by Hartman in \cite{Hartman75}.

A similar type of instability was also 
recently observed in the compressible case \cite{ADM20} and in \cite{Zillinger20} in the simplified case of a 
viscous fluid under the Boussinesq approximation.
For homogeneous fluids linearized around Couette, a completely different (linear) model was introduced in \cite{DZ19}, where the authors employ \emph{ad hoc} initial data, which are built in such a way that the vorticity is unstable in any $H^s$, $s>2$ but the velocity still enjoys inviscid damping. This is related to the so-called mechanism of \textit{echoes}.

\subsection{Proof strategy and comparison with previous works}
To the best of our knowledge, a linear asymptotic stability analysis of shear flows in the non-constant density setting 
appears in \cite{Lin18} and \cites{ADM20}.

In \cite{Lin18}, the authors treat the case of the Couette flow $U(y)=y$
and an exponential density 
profile as the one in \eqref{eq:density}, for an arbitrary $R>0$. In the case $R>1/4$, an improvement of the 
statement of \cite{Lin18} is recovered by Theorem
\ref{thm:CloseCouette} by simply setting $\eps=0$. In particular no (exponential) weights are needed in the norms. 
However, it is shown in \cite{Lin18} that for $R\leq 1/4$ the decay rates  undergo (at least) a
logarithmic correction. The methods of proof in our work are completely different from those employed in  \cite{Lin18}. 
There, the analysis is based on  the use of hypergeometric functions (an approach that dates back to Hartman \cite{Hartman75}), 
which were needed to solve the system 
\eqref{eq:EulerBomega} in frequency space. In fact, by performing the change of variables to follow the background shear, it is possible to rewrite \eqref{eq:EulerBomega} as a second order linear ODE in terms of $\psi$ that can be solved with the use of hypergeometric functions. 
	
In our paper, we instead take inspiration from the recently treated compressible case  \cite{ADM20}. Namely, we consider \eqref{eq:EulerBomega} as a $2\times 2$ non-autonomous dynamical system at each fixed frequency. Then, by exploiting a proper symmetrization, we present a proof of Theorem \ref{thm:CloseCouette} in the Couette case (i.e. for $\eps=0$), which will be  propaedeutic to the study of shears
close to Couette. 

Finally, our method of proof relies on the use of an energy functional endowed with certain suitable weights, among which one is inspired by the Alinhac ghost weight energy method \cite{Alinhac01} and was already used by Zillinger (see \cite{Zillinger14} and references therein) in the constant density case, and in \cite{ADM20} in the compressible case.
Notice that the analysis of this model requires a delicate tracking of
the constants depending on $\beta$, which will be involved in the definition of one of the weights of our energy functional. In fact, this is the idea which allows us to handle the general case of exponentially stratified fluids near Couette.
We would also like to point out that the shears-near-Couette case could be seen as a step forward aiming at investigating \emph{nonlinear} stability of the Couette flow for stratified fluids, as a good understanding of the dynamics of the background shear flow was a key point for the nonlinear stability in the constant density case, see \cite{BM15}.

\subsection{Stratified flows in the Boussinesq approximation}\label{sub-1.1}
The analysis of non-homogeneous fluids is commonly carried out under the \textit{Boussinesq approximation}, namely ``density is assumed constant except when it directly causes buoyant forces"-\cite{gray1976validity}. In particular, the system  \eqref{eq:nonlin} is reduced to 
\begin{equation}
\label{syst:bouss}
\begin{split}
		\bar{\rho}_{c}(\dt+\uu \cdot \nabla) \uu+\nabla \widetilde{p}&=-\widetilde{\rho}\boldsymbol{g},\\
		(\dt+\uu \cdot \nabla)\widetilde{\rho}&=0,
\end{split}
\end{equation}
where  $\bar{\rho}_{c}$ is the constant density average. For a physical derivation of this approximation we refer to \cite{gray1976validity}. It is then customary to assume that at equilibrium the stratification (or density variation) is given by a \emph{stable} profile $\bar \rho (x, y) = \bar \rho(y)$ so that $\de_y\bar \rho (y) <0$. 
Among all the possible stratification's equilibria, one often assumes in addition that the stratification is
locally affine so that $\de_y\bar{\rho}(y)$ is constant, see \cite{BDSR, Lannes-Saut} and references therein for an interesting discussion on the benefits and limitations of this approximation. This way, we want to linearize the system \eqref{syst:bouss} around the equilibrium 
\begin{align}
	\label{def:steadybouss}
	\bar\rho=\bar\rho(y)=1-\gamma y, \qquad \de_y\bar p=-\mathfrak{g} \bar\rho, \qquad \bar{\uu}=(U(y),0).
\end{align}
We just remark that when $U(y)=0$ it describes perturbations of a steady solution where gravity compensates for the vertical derivative of the pressure ($\partial_y \bar{p}=-\mathfrak{g} \bar{\rho}$), which translates in the fact that, in this Boussinesq regime, the restoring force of equilibrium's fluctuations is exactly gravity (\emph{Archimedes' principle}). For a well-posedness and stability theory of small (and finite energy) perturbations of stratified fluids with zero velocity at the equilibrium see for instance \cite{Charve17, Danchin, BDSR,CCL18}. \\ \indent 
Thus, we consider  
\begin{equation}
	\uu=\vv+\bar{\uu}, \qquad \widetilde{\rho}=\rho+\bar{\rho}, \qquad \widetilde{p}=p+\bar p.
\end{equation}
 We define 
\begin{align}\label{def:scaled-density-Bouss}
	\omega=\nabla^\perp \cdot \vv, \qquad q=\dfrac{\rho}{\gamma}, \qquad R=\gamma \mathfrak{g},
\end{align}
where $\omega$ and $q$ are, respectively, the vorticity and the scaled density of the perturbation whereas $R$ is the \textit{Brunt-V\"ais\"al\"a frequency}. By assuming $\bar{\rho}_c=1$ the linearization of the system \eqref{syst:bouss} around the \textit{linearly stratified shear flow} \eqref{def:steadybouss} is given by 
\begin{equation}
	\begin{aligned}\label{eq:bouss-last}
		(\de_t + U(y) \de_x)\omega -  U''(y) \de_x \psi&=-R \de_x q,\\
		(\de_t + U(y) \de_x)q&=  \de_x \psi,\\
		\Delta \psi &=\omega.
	\end{aligned}
\end{equation}	
Notice that \eqref{eq:bouss-last} is exactly \eqref{eq:EulerBomega} with $\beta=0$, so that Theorem \ref{thm:CloseCouette} gives asymptotic
stability for system \eqref{eq:bouss-last} in the regime $R>1/4$.

\subsection{The stably stratified regime}\label{sub-1.2}
A model for the stably stratified regime can be heuristically derived from the Euler equations \eqref{eq:nonlin} as follows.
We consider an \emph{exponential stratification}, where the \emph{stable} background density is $\bar{\rho}(y)=\e^{-\beta y}$ for some $\beta>0$,
and we linearize system \eqref{eq:nonlin} around this steady profile. Again, we are interested in studying stability of shear flows. Putting both the approximations together, we thus linearize \eqref{eq:nonlin} near

\begin{align}
\bar\uu=(U(y),0), \qquad \bar\rho=\bar\rho(y)=\e^{-\beta y}, \qquad \de_y\bar p=-\mathfrak{g} \bar\rho.
\end{align}
Hitting the equation for the velocity with $\nabla^\perp\cdot$, we obtain
\begin{equation}\label{eq:linearization-intro}
	\begin{aligned}
	(\de_t + U(y)\de_x) \omega - U''(y)\de_x\psi+\bb(\de_tv^x + U(y)\de_x v^x + U'(y)\de_x\psi)&=-\mathfrak{g}\frac{\de_x\rho}{\bar\rho},\\
	(\de_t+ U(y)\de_x)\frac{\rho}{\bar\rho}&=\bb\de_x\psi,\\
	\Delta \psi&=\omega,
	\end{aligned}
\end{equation}
where the velocity perturbation is recovered again from the Biot-Savart law in \eqref{eq:biot}. 
We rearrange now the linearized system \eqref{eq:linearization-intro} by introducing the scaled density and the parameter $R$
\begin{align}\label{eq:density}
q&=\frac{\rho}{\beta \bar{\rho}}, \qquad R=\beta \mathfrak{g}.
\end{align}
Using that $v^x=-\de_y\psi$, we recover the equations of system \eqref{eq:EulerBomega}. Due to the (linear) relation between $\beta$ and $R$,
our Theorem \ref{thm:CloseCouette} provides an answer to the question of asymptotic stability in the case $\beta > \frac{1}{4\mathfrak{g}}$, or equivalently $R>\frac{1}{4}$. 
We point out that, besides stability of the stratification profile, the positivity of $\beta$ is actually needed to satisfy the criterion $R=\beta \mathfrak{g} > \frac{1}{4}$, which in our strategy turns to be the key condition in order to have a coercive energy functional.

\begin{remark}
The threshold $R>1/4$ is linked to the celebrated \emph{Miles-Howard criterion} for (spectral) stability of stratified fluids
\cites{Howard61,Miles61}, which precisely requires the \emph{Richardson number}  
\begin{align}
\mathrm{Ric}(y)=-\mathfrak{g} \frac{\bar{\rho}'(y)}{\bar{\rho}(y)} \frac{1}{|U'(y)|^2}= \frac{R}{|U'(y)|^2}
\end{align}
to be greater than $1/4$. 
Notice that in the case of shear flows close to Couette (see Theorem \ref{thm:CloseCouette}), this condition can be violated since $|U'(y)|<1+\eps$ only implies $\mathrm{Ric}(y)>1/4-\eps$.
This partially explains the $\eps$-losses in the decay rates, in agreement with what can be proved in the Couette case \cite{Lin18}.
\end{remark}

\subsection{Organization of the paper}

The next Section \ref{sec:prelim} is dedicated to the set-up of the functional-analytic tools needed to study our problem.
In Section \ref{sec:shearsCouette}, relying on a dynamical system approach, we prove 
optimal inviscid damping decay rates of \eqref{eq:EulerBomega} for the Couette flow $U(y)=y$. 
Next, we are able to extend the result (up to a small loss in terms of the decay rates) to shears of general profiles, which are close to Couette in a suitable sense. The latter result is obtained by means of a different method, which is based on the use of a properly weighted energy functional, and is presented in Section \ref{sec:shearsClose}.

\section{Preliminaries}\label{sec:prelim}
In this section, we introduce the set of variables, the functional spaces and the operators which will be used in our analysis. First, we define new coordinates and variables, which follow the background shear flow and will be sometimes referred as the ``moving frame''. Next, we set the notation for the Fourier transform and the functional spaces. Since the new spatial coordinates as well as partial derivatives depend on derivatives of the shear flow, then the action of differential operators in the moving coordinates has to be established and well-defined. Related results are provided in Proposition \ref{prop-delta-t}-\ref{prop:Bt} at the end of this section.

\subsection{Change of coordinates and decoupling}
It is best to further re-write the system \eqref{eq:EulerBomega} by defining the auxiliary variable
\begin{align}\label{eq:relomegatheta}
\theta=\omega -\beta\dy \psi=(I-\beta\de_y\Delta^{-1})\omega.
\end{align}
Since the operator $I-\beta\de_y\Delta^{-1}$ is well-defined and invertible for any $\beta \ge 0$,
we can write  \eqref{eq:EulerBomega}  in terms of $\theta$ and $q$ alone as
\begin{equation}\label{eq:EulerB}
\begin{aligned}
(\de_t + U(y)\de_x)\theta &=-\RR\de_xq+(U''(y)-\beta U'(y))\de_x\Delta^{-1}\left(I-\beta\de_y\Delta^{-1}\right)^{-1}\theta,\\
(\de_t+ U(y)\de_x)q&=\de_x\Delta^{-1}\left(I-\beta\de_y\Delta^{-1}\right)^{-1}\theta,
\end{aligned}
\end{equation}
with initial data
\begin{align}
\theta(0,x,y)=\theta_{in}(x,y),\qquad q(0,x,y)=q_{in}(x,y).
\end{align}
Due to the transport nature of \eqref{eq:EulerB}, the convenient set of coordinates is that of a moving frame that follows the background shear. Define
\begin{align}
X=x-U(y)t,\qquad Y=U(y)
\end{align}
and the corresponding new unknowns 
\begin{equation}\label{def:new-variables}
\Theta(t,X,Y)=\theta(t,x,y),\quad \QQ(t,X,Y)=q(t,x,y),\quad \Omega(t,X,Y)=\omega(t,x,y).
\end{equation}
The differential operators will change accordingly, in analogy with the cases already studied in the incompressible literature \cite{Zillinger14,Zillinger15,CZZ18, JIA19}. 
In particular, by defining 
\begin{align}\label{def-g-b}
g(Y):=U'(U^{-1}(Y)),\qquad b(Y):=U''(U^{-1}(Y)),
\end{align}
we obtain the rules
\begin{align}
\de_x\rightsquigarrow\de_X,\qquad \dy \rightsquigarrow g(Y)(\dY-t\dX),
\end{align}
and
\begin{align}\label{def:Deltat}
\Delta\rightsquigarrow \Delta_t:=\dXX+g^2(Y)(\dY-t\dX)^2+b(Y)(\dY-t\dX).
\end{align}
As we shall see, if $g,b$ are small in a way made more precise later (see Proposition \ref{prop-delta-t}), the operator $\Delta_t$ is invertible.
Similarly (see Proposition \ref{prop:Bt}), the operator
\begin{align}
\label{def:Bt}B_t=&\left(I-\beta g(Y)(\dY-t\dX)\Delta_t^{-1}\right)^{-1}.
\end{align}
is well-defined for any $\beta$. Moreover, it follows from \eqref{eq:relomegatheta} and \eqref{def:new-variables} that
\begin{equation}
\Theta:=B_t^{-1}\Omega.
\end{equation}
In the moving frame, the equations \eqref{eq:EulerB}, written in terms of $\Theta$ and $\QQ$ read  
\begin{equation}\label{eq:EulerBmove}
\begin{aligned}
 \dt \Theta &=-\RR\dX Q+\left(b(Y)-\beta g(Y)\right)\dX \Delta_t^{-1}B_t\Theta,\\
\dt \QQ&=\dX \Delta_t^{-1}B_t\Theta.
\end{aligned}
\end{equation}
It is apparent that the above system decouples in the $X$-frequency. Precisely, we adopt the following convention. Given a function
$f=f(X,Y)$, we can write
\begin{align}\label{def:F-transform-x}
f(X,Y)=\sum_{k\in \ZZ} f_k(Y)\e^{ik X}, \qquad f_k(Y)=\frac{1}{2\pi}\int_\T f(X,Y)\e^{-ik X}\dd X,
\end{align} 
where throughout the paper $k\in \ZZ$ will denote the $X$-Fourier variable.
Applying this reasoning to $\Omega$ and $Q$, we can interpret \eqref{eq:EulerBmove} as a coupled system of infinitely
many equations that read
\begin{equation}\label{eq:EulerBmovek}
\begin{aligned}
 \dt \Theta_k&=-ik\RR Q_k+ik\left(b(Y)-\beta g(Y)\right) \Delta_t^{-1}B_t\Theta_k,\\
\dt \QQ_k&=ik \Delta_t^{-1}B_t\Theta_k.
\end{aligned}
\end{equation}
In the above equations, we identified $\Delta_t^{-1}$ and $B_t$ with their $X$-Fourier localizations.
It is clear that for $k=0$ there is nothing to prove, and therefore we will always assume $k\neq 0$.

\subsection{Function spaces and operators}
The Fourier transform in $Y$ of a function $f=f(X,Y)$ is defined by
\begin{equation}\label{def:Fourier-transform}
\ff(X,\eta)=\int_{\mathbb{R}} f(X,Y) \e^{-i\eta Y} \dd Y, \qquad \eta\in \R.
\end{equation}
For any $k\in \ZZ$ and and $\eta\in \R$, we define
\begin{align}
\l \eta\r=\sqrt{1+\eta^2},\qquad \l (k, \eta)\r=\sqrt{1+k^2+\eta^2}.
\end{align}
We define the 
Sobolev space $H^s$ of order $s\in \R$ by the scalar product
\begin{align}
\l f,h\r_s=\l \widehat{f},\widehat{h}\r_s=\sum_{k\in \ZZ}\int_{\R}  \l (k, \eta)\r^{2s}\ff_k(\eta)\overline{\widehat{h}_k(\eta)}\dd\eta
\end{align}
and norm
\begin{align}
\| f\|_s^2=\| \widehat{f}\|_s^2=\sum_{k\in \ZZ}\int_{\R}  \l (k, \eta)\r^{2s}|\ff_k(\eta)|^2\dd\eta.
\end{align}
By the above notation, we mean that we make tacit use of Plancherel theorem.
When $s=0$, the above reduce to the standard $L^2$ scalar product and norm, and we will use the explicit notation $\| \cdot\|_{L^2}$ for the norm.
An important role will be played by the operator
\begin{align}
\Delta_L:=\dXX+(\dY-t\dX)^2,
\end{align}
which is exactly what one obtains from $\Delta_t$ in \eqref{def:Deltat} when setting $g\equiv1$ and $b\equiv0$.
The Fourier symbol of $-\Delta_L$ will be denoted by
\begin{align}\label{def:p}
p(t;k,\eta)=k^2+(\eta-k t)^2,
\end{align}
along with its time-derivative
\begin{align}
p'(t;k,\eta)=-2k(\eta-k t).
\end{align}
To see $\Delta_t$ as a perturbation of $\Delta_L$, it is convenient to write
\begin{equation}\label{eq:DeltatI-T}
\Delta_t=\left[I-T_{\eps}\right]\Delta_L,
\end{equation}
with
\begin{align}\label{def:tildeTL}
T_{\eps}=\left[(g^2-1)(\dY-t\dX)^2+b(\dY-t\dX)\right](-\Delta_L^{-1}).
\end{align}
In particular,
\begin{align}\label{def:tildeTLbg}
T_{\eps}=T_{\eps}^g+ T_{\eps}^b,
\end{align}
where
\begin{align}\label{eq:Tgandb}
T_{\eps}^g=(g^2-1)(\dY-t\dX)^2(-\Delta_{L}^{-1}), \qquad T_{\eps}^b=b(\dY-t\dX)(-\Delta_{L}^{-1}).
\end{align}
The following proposition holds true. 
\begin{proposition}\label{prop-delta-t}
Let $s\geq 0$ be arbitrarily fixed. There exists a constant $c_s\geq 1$ with the following property.
Assume that 
\begin{align}
\norma{g-1}{s+1}+\norma{b}{s+1}\leq \eps,
\end{align}
 for a positive $\eps< 1/c_s$. Then $\Delta_t$ is invertible on $H^s$ and
\begin{equation}\label{def:Deltat-1}
\Delta_t^{-1}=\Delta_L^{-1} T_L,
\end{equation}
where 
\begin{equation}\label{def:TL}
T_L=\sum_{n=0}^{\infty}T_{\eps}^n=I+T_\eps T_L.
\end{equation}
The convergence above is in the $H^s$-operator norm, and
\begin{align}\label{eq:normbdneum}
\norm{T_L}_{H^s\to H^s} \leq \frac{1}{1-c_s\eps}.
\end{align} 
\end{proposition}

\begin{proof}
We begin to prove the convergence of the Neumann series \eqref{def:TL} by obtaining a suitable bound on
\begin{align}
\norm{T_{\eps}f}_{s}\leq \|T_\eps^g f\|_{s}+\|T_\eps^b f\|_{s},
\end{align}
for any $f:\R\to\C$ in $H^s$.
To estimate the first part, notice that the Fourier symbol of the operator
$(\dY-t\dX)^2(-\Delta_{L}^{-1})$ is uniformly bounded above by 1. Therefore, by Young's convolution inequality and the fact that
$g\in L^\infty$, we have
\begin{align}
\norm{T_\eps^g f}_{s} &\leq \norm{|\widehat{g^2-1}|*|\ff|}_{s}\leq \norm{\l\cdot\r^s\widehat{g^2-1}}_{L^1}\| \l\cdot\r^s f\|_{L^2}\lesssim \norm{g-1}_{s+1}\|f\|_{s},
\end{align}
where the last inequality follows from
\begin{align}
 \norm{\l\cdot\r^s\widehat{g^2-1}}_{L^1}=\int_\R \l\eta\r^s\left|\widehat{g^2-1}\right| \dd \eta =\int_\R \frac{1}{\l\eta\r}\l\eta\r^{s+1}\left|\widehat{g^2-1}\right| \dd \eta
 \lesssim \norm{g^2-1}_{s+1} \lesssim \norm{g-1}_{s+1}.
\end{align}
In the same way,
\begin{align}
\|T_\eps^b f\|_{s}\lesssim \norm{b}_{s+1}\|f\|_{s}.
\end{align}
Hence, there exists a constant $c_s$ such that
\begin{align}
\label{bd:TtildeL}
\norm{T_{\eps}f}_{s}\leq c_s\left(\norm{g-1}_{s+1}+\norm{b}_{s+1}\right)\|f\|_{s}\leq c_s \eps \norm{f}_s.
\end{align}
If $\eps$ is chosen such that $c_s \eps<1$,
the Neumann series \eqref{def:TL} converges. In particular,
\begin{align}
T_L=[I-T_{\eps}]^{-1},
\end{align}
so that from \eqref{eq:DeltatI-T} we also obtain the relation \eqref{prop-delta-t}, and the proof is over.
\end{proof}

We also need to describe the action of $B_t$ on $H^s$.  Its definition relies on the operator
\begin{align}\label{def:BL}
B_L:=(I-\beta (\de_Y-t\de_X)\Delta_L^{-1})^{-1},
\end{align}
which is precisely $B_t$ in the Couette case where $g\equiv1$ and $b\equiv 0$ (compare with \eqref{def:Bt}). More precisely, adopting the same notation as for the expansion of $\Delta_t^{-1}$ in \eqref{def:Deltat-1}-\eqref{def:TL}, we prove in the following proposition that we can write

\begin{align}\label{def:2Bt}
	B_t=T_{B} B_L ,
\end{align}
where $T_{B}$ is a proper Neumann series (see Proposition \ref{prop:Bt}).
Notice that the operator $B_L$ is
a Fourier multiplier, and
\begin{align}\label{eq:BLsymb}
B_L^{-1}=1+\dfrac{i\beta(\eta-kt)}{p},
\end{align}
so everything is well-defined. 
It is now an easy task to show that, for any fixed frequency, the following bound holds true
\begin{align}\label{eq:BLboundbelow}
\dfrac{1}{\sqrt{1+\beta^2}}\le \left|B_L (t; k, \eta) \right|\le 	1.
\end{align} 
For further reference, we also need to introduce
\begin{align}\label{def:Beps}
B_\eps:&=B_L(\beta (g-1)(\de_Y-t\de_X)\Delta_L^{-1}+\beta g(\de_Y - t \de_X) \Delta_L^{-1}T_\eps T_L).
\end{align}
The action of the above operators on $H^s$ is detailed in the proposition below.
\begin{proposition}
	\label{prop:Bt}
	Let $s\geq 0$ and $\beta \ge 0$ be arbitrarily fixed. 
Assume that
\begin{align}\label{ass:small}
\norma{g-1}{s+1}+\norma{b}{s+1}\leq \eps,
\end{align}
 for a positive $\eps \in (0, \eps_0]$. Then we have

\begin{align}
\label{bd:TildeB}\norm{B_\eps}_{H^s\to H^s} \lesssim \eps.
\end{align} 
The operator $B_t$ on $H^s$ is well-defined and explicitly given by
\begin{align}\label{eq:NeumBt}
B_t&=(I-B_\eps)^{-1}B_L=T_{B} B_L,
\end{align}
where
\begin{align}\label{def:TB}
T_{B}=\sum_{n=0}^\infty B_\eps^n=I+B_\eps T_{B}.
\end{align}
Moreover,
\begin{align}\label{bd:norm_TBL}
\norm{T_{B}}_{H^s \to H^s} \le 2.
\end{align}
\end{proposition}

\begin{proof}[Proof of Proposition \ref{prop:Bt}]
The expansion of $B_t$ in \eqref{def:Bt} in terms of the Neumann series \eqref{eq:NeumBt}
follows from the definition of $\Delta_t^{-1}$ in \eqref{def:Deltat-1}-\eqref{def:TL} and the identity
\begin{align}\label{eq:last-Bstar}
B_t&=(I-\beta g (\de_Y-t\de_X)\Delta_t^{-1})^{-1}\notag\\
&=(I-\beta g (\de_Y-t\de_X)(\Delta_L^{-1}+\Delta_L^{-1}T_{\eps}T_L))^{-1}\notag\\
&=(I-\beta (g-1)(\de_Y-t\de_X)\Delta_L^{-1}-\beta (\de_Y-t\de_X)\Delta_L^{-1}  -\beta g  (\de_Y - t \de_X) \Delta_L^{-1}T_{\eps}T_L)^{-1}\notag\\
&=(B_L^{-1} (1-B_\eps))^{-1}=(1-B_\eps)^{-1}B_L=T_{B}B_L,
\end{align}
where $B_L, B_\eps$ are in \eqref{def:BL} and \eqref{def:Beps}. The above expansion is rigorous provided that $(I-B_\eps)$ is invertible.
The expression of $B_\eps$ in \eqref{def:Beps} is rearranged here as follows
\begin{align}\label{def:Beps1}
B_\eps&=B_L(\beta (g-1)(\de_Y-t\de_X)\Delta_L^{-1}+\beta (g-1)(\de_Y - t \de_X) \Delta_L^{-1}T_\eps T_L+\beta (\de_Y - t \de_X) \Delta_L^{-1}T_\eps T_L).
\end{align}
We use the upper bound on the multiplier $B_L$ in \eqref{eq:BLboundbelow} and the fact that $|\eta-kt|p^{-1} \le p^{-1/2} \le 1$ to obtain
\begin{align}\notag
\norm{B_\eps f}_s & \le \beta \norm{ |\widehat{g-1}|*|\widehat{f}|}_s + \beta \norm{|\widehat{g-1}|* |\widehat{T_{\eps}T_L  f}|}_s + \beta \norm{T_\eps T_L f}_s\\
& \lesssim \beta \norm{\l \cdot \r^s \widehat{g-1}}_{L^1} \norm{f}_s + \norm{\l \cdot \r^s \widehat{g-1}}_{L^1} \norm{T_\eps T_L f}_s+ \beta \norm{T_\eps T_L f}_s\notag\\
& \lesssim  \beta \norm{g-1}_{s+1} \norm{f}_s + \beta (1+\norm{g-1}_{s+1}) \norm{T_\eps T_L f}_s,
\end{align}
where we applied the Young convolution inequality exactly as in the proof of Proposition \ref{prop-delta-t}. It remains to deal with the last term above. Thus we first appeal to \eqref{bd:TtildeL} and after we employ \eqref{eq:normbdneum}, so that
\begin{align*}
\norm{T_\eps T_Lf}_s & \le c_s (\norm{g-1}_{s+1}+\norm{b}_{s+1})\norm{T_L f}_{s}
\le \dfrac{c_s}{1-c_s \eps_0} (\norm{g-1}_{s+1}+\norm{b}_{s+1})\norm{f}_{s},
\end{align*}
which yields \eqref{bd:TildeB} thanks to the smallness assumptions \eqref{ass:small}, where $\eps_0 \ge \eps$ satisfies 
\begin{align*}
\ \beta \eps_0 \left(1+\dfrac{4  c_s }{1-c_s \eps_0}\right) < 1,
\end{align*}
so that $T_B$ in \eqref{def:TB} is well-defined and $(I-B_\eps)$ is invertible. The bound \eqref{bd:norm_TBL} directly follows from its definition \eqref{def:TB} and the bound on $B_\eps$.
The proof is over.
\end{proof}

\subsection{Notation and conventions}
We highlight the following conventions, which will be used throughout the paper.
\begin{itemize}
\item As already pointed out before, $p'(t; k, \eta)$ stands for the time-derivative of the Fourier symbol $p(t; k, \eta)$. This notation is systematically applied to the Fourier symbols of the weights involved in our $H^s$-weighted norm.
\item We occasionally identify operators with their symbols in order to avoid additional notations (this is for instance the case of the operator $\Delta_t^{-1}$, as pointed out in \eqref{eq:EulerBmovek} and the lines below).
\item We drop the subscript $k$ related to the $X$-Fourier-localization of the variables defined in \eqref{def:F-transform-x} and we also identify functions with their $Y$-Fourier transform in \eqref{def:Fourier-transform} when there is no confusion. The adoption of this convection is usually pointed out at the beginning of the proofs.
\item We use the notation $f_1\lesssim f_2$ if there exists a constant $C=C(R,\beta,s)$ such that $f_1\leq Cf_2$,  where $C$ is independent
of $k$, $\eps$ and $\beta$. We denote $f_1\approx f_2$ if $f_1\lesssim f_2$ and $f_2\lesssim f_1$.
\end{itemize}

\section{The Couette flow}\label{sec:shearsCouette}
We begin our analysis with the Couette flow, when $U(y)=y$. In this case, in \eqref{def-g-b} we have $g\equiv 1$ and $b\equiv 0$, so that $B_t$ in \eqref{def:Bt} coincides with $B_L$ in \eqref{def:BL}, so that 
system \eqref{eq:EulerBmovek} takes the simpler form
\begin{equation}\label{eq:couettek}
\begin{aligned}
 \dt \Theta_k&=-ik\RR Q_k-ik\beta \Delta_L^{-1}B_L\Theta_k,\\
\dt \QQ_k&=ik \Delta_L^{-1}\Theta_k+ik \Delta_L^{-1}(B_L-1)\Theta_k,
\end{aligned}
\end{equation}
where
\begin{align}\label{def:theta-Couette}
\Theta_k=B_L^{-1}\Omega_k.
\end{align}
We remark again that $B_L$ is just a Fourier multiplier, and hence commutes with $\Delta_L^{-1}$. One can explicitly write its symbol and the symbol of its inverse operator, while the lower and upper bounds were already provided in \eqref{eq:BLboundbelow}.

System \eqref{eq:couettek} is a non-autonomous dynamical system for each fixed frequency $(k,\eta)$. Although the proof carried out
in the next Section \ref{sec:shearsClose}
for more general shears applies in this case as well, we consider a different point of view,
and argue \emph{point-wise} in both $k$ and $\eta$. 
From now on, we slightly abuse notation and identify $\Theta_k$ and $Q_k$ with their $Y$-Fourier transforms $\widehat{\Theta}_k$ and $\widehat{Q}_k$. Our result reads as follows.
\begin{theorem}\label{thm:Couette-main}
Let $\beta\geq 0$ and $k\neq 0$.  The solution to \eqref{eq:couettek} satisfies the uniform bounds
\begin{align}\label{eq:estimatecouette}
 |p^{-\frac14}\Theta_k(t)|^2+|p^{\frac14}Q_k(t)|^2\approx  |(k^2+\eta^2)^{-\frac{1}{4}} \Theta_k(0)|^2 + |(k^2+\eta^2)^\frac{1}{4} Q_k(0)|^2,\qquad \forall t\geq 0,
\end{align}
point-wise in $\eta\in \R$. In particular, thanks to \eqref{eq:BLboundbelow},
\begin{align}\label{eq:estimatecouette2}
 |p^{-\frac14}\Omega_k(t)|^2+|p^{\frac14}Q_k(t)|^2\approx  |(k^2+\eta^2)^{-\frac{1}{4}} \Omega_k(0)|^2 + |(k^2+\eta^2)^\frac{1}{4} Q_k(0)|^2,\qquad \forall t\geq 0,
\end{align}
point-wise in $\eta\in \R$.
\end{theorem}
The pre-factors $p^{\pm\frac14}$ in \eqref{eq:estimatecouette2} appear in view of a natural symmetrization of 
system \eqref{eq:couettek} that we carry out in the next Section \ref{sub:Couettesym}.  
Now we show that the proof of the main Theorem \ref{thm:CloseCouette} in the case the Couette flow (namely, $\eps=0$) follows from \eqref{eq:estimatecouette}.

\begin{proof}[Proof of Theorem \ref{thm:CloseCouette} for Couette]

Recall that from \eqref{eq:biot}  the velocity field (in the moving fra\-me) is given by
\begin{align}\label{def:velocity}
\boldsymbol{V}_k=(V^x_{k}, V^y_{k})
=(-(\partial_Y -t\partial_X)\Delta_L^{-1} \Omega_k, \partial_X \Delta_L^{-1} \Omega_k).
\end{align}
Thus, using that $|\eta-kt|^2 p^{-1} \le 1$,  from \eqref{eq:estimatecouette2} we deduce 

\begin{align}\label{estimate-u}
\|V_{k}^x(t)\|_{L^2}^2 & = \int_\mathbb{R} \dfrac{|\eta-kt|^2}{p^\frac{3}{2}\langle (k, \eta)\rangle} \, |p^{-\frac{1}{4}}\Omega_k(t)|^2 \, \langle (k, \eta) \rangle \dd\eta  \leq \int_\mathbb{R} \dfrac{1}{p^\frac{1}{2}\langle (k, \eta)\rangle} \, |p^{-\frac{1}{4}}\Omega_k(t)|^2 \, \langle (k, \eta) \rangle  \dd\eta \notag\\
&\lesssim \frac{1}{\langle t \rangle} \left[\frac{1}{|k|^2}\|\Omega_k(0)\|_{L^2}^2+\|Q_k(0)\|_{1}^2\right].
\end{align}
Similarly,

\begin{align}\label{estimate-v}
\|V_{k}^y(t)\|_{L^2}^2 &\lesssim \frac{1}{\langle t \rangle^3}\left[\frac{1}{|k|^2}\|\Omega_k(0)\|_{1}^2+\|Q_k(0)\|_{2}^2\right],
\end{align}
and 

\begin{align}\label{estimate-q}
\|Q_k(t)\|_{L^2}^2 &\lesssim \frac{1}{\langle t \rangle} \left[\frac{1}{|k|^2}\|\Omega_k(0)\|_{L^2}^2+\|Q_k(0)\|_{1}^2\right].
\end{align}
Thus, Theorem \ref{thm:CloseCouette} is proven.

\end{proof}

We can also give a proof of Corollary \ref{cor:Lyinst} directly from Theorem \ref{thm:Couette-main}.

\begin{proof}[Proof of Corollary \ref{cor:Lyinst}]

	Just by looking at the lower bound in \eqref{eq:estimatecouette2}, we have  
	\begin{align}\label{bd:OmkC}
	\|\Omega_k(t)\|_{L^2}^2+\|p^\frac12 Q_k(t)\|_{L^2}^2
	&= \int_\mathbb{R} p^\frac{1}{2}  \left[ |p^{-\frac{1}{4}}\Omega_k(t)|^2  + |p^{\frac{1}{4}}Q_k(t)|^2 \right]   \dd\eta\notag\\ 
	&\gtrsim \l t\r \left[\|\Omega_k(0)\|_{-1}^2+|k|^{2}\|Q_k(0)\|_{L^2}^2\right],
	\end{align}
	which proves the lower bound. The upper bound is similar.
\end{proof}

\subsection{Symmetric variables}\label{sub:Couettesym}

Since \eqref{eq:couettek} decouples in the $X$-Fourier variable, we fix a nonzero integer $k$ and define the auxiliary variables 
\begin{equation}\label{eq:Z1Z2couette}
Z_1:= p^{-\frac{1}{4}} \Theta_k, \qquad Z_2:=p^\frac{1}{4}i \sqrt{R} Q_k.
\end{equation}
Taking into account \eqref{eq:NeumBt}, we then find
\begin{align}\label{eq:symm-matrixCouette}
\de_t\begin{pmatrix}
Z_1\\
Z_2
\end{pmatrix}
=\begin{pmatrix}
-\dfrac{1}{4}\dfrac{p'}{p} & -k\sqrt{R}p^{-\frac{1}{2}}\\
k\sqrt{R}p^{-\frac{1}{2}}  & \dfrac{1}{4}\dfrac{p'}{p}
\end{pmatrix}
\begin{pmatrix}
Z_1\\
Z_2
\end{pmatrix}
+\begin{pmatrix}
\beta\dfrac{ik}{p} B_L & 0\\
k\sqrt{R}p^{-\frac{1}{2}}(B_L-1)  & 0
\end{pmatrix}
\begin{pmatrix}
Z_1\\
Z_2
\end{pmatrix}.
\end{align}
The first matrix in \eqref{eq:symm-matrixCouette} is what determines most of the behavior of the system, while the second one is
considered as a remainder (i.e. time-integrable). To exploit this structure,
we define the point-wise energy functional as 
\begin{align}\label{def:pointwise-functional-Couette}
E(t)&=\frac12\left[|Z_1(t)|^2+|Z_2(t)|^2+\frac{1}{2k\sqrt{R}} \Re  \left(p' p^{-\frac12} Z_1(t) \overline{Z_2(t)}\right)\right].
\end{align}
The last term is simply coming from the ratio between the diagonal and non-diagonal entries in the first matrix in \eqref{eq:symm-matrixCouette}.
Now using that $|p'|p^{-\frac12}\leq 2|k|$, we deduce that
\begin{align}\label{eq:estimate-mixed-termCouette}
\frac{1}{2|k|\sqrt{R}} \left| p' p^{-\frac12} Z_1 \overline{Z_2} \right|\leq\frac{1}{\sqrt{R}}|Z_1||Z_2|\leq\frac{1}{2\sqrt{R}}\left(|Z_1|^2+|Z_2|^2\right).
\end{align}
As a consequence,  the functional $E$ is coercive whenever $R>1/4$, namely
\begin{align}\label{eq:coercive-pointwise}
\frac12\left(1-\frac{1}{2\sqrt{R}}\right)\left[|Z_1|^2+|Z_2|^2\right]\leq E\leq\frac12\left(1+\frac{1}{2\sqrt{R}}\right)\left[|Z_1|^2+|Z_2|^2\right].
\end{align}

\begin{remark}
It is quite
interesting  that the Miles-Howard criterion is connected precisely to the positive definiteness of the quadratic form \eqref{def:pointwise-functional-Couette}
arising \emph{after} symmetrization. In the classical proof \cite{Howard61}, this is derived by a suitable conjugation that ensures that $R>1/4$ is a sufficient
condition for spectral stability.
\end{remark}

\subsection{Proof of Theorem \ref{thm:Couette-main}}
In account of the coercivity of $E(t)$, the proof of Theorem \ref{thm:Couette-main} can be obtained by estimating $E(t)$, which will be done via a Gr\"onwall estimate. Hence, we need to compute the time-derivative of $E$.
We first notice that
\begin{align}
\frac12\ddt |Z_1|^2=\Re(\de_tZ_1 \overline{Z_1})= -\dfrac{1}{4}\dfrac{p'}{p} |Z_1|^2-\beta\dfrac{k}{p}\Im(B_L) |Z_1|^2- k\sqrt{R}p^{-\frac{1}{2}}\Re(Z_1\overline{Z_2}).
\end{align}
Similarly,
\begin{align}
\frac12\ddt |Z_2|^2= \dfrac{1}{4}\dfrac{p'}{p} |Z_2|^2 +k\sqrt{R}p^{-\frac{1}{2}} \Re(Z_1\overline{Z_2})+k\sqrt{R}p^{-\frac{1}{2}} \Re((B_L-1)Z_1\overline{Z_2}).
\end{align}
Regarding the last term, 
we  find that
\begin{align}
\ddt\Re  \left(p' p^{-\frac12} Z_1 \overline{Z_2}\right)
&=\left(p' p^{-\frac12}\right)'\Re\left(Z_1\overline{Z_2}\right)+\beta k\frac{p'}{p^\frac32} \Re\left(i B_L Z_1\overline{Z_2}\right)\notag\\
&\quad- k\sqrt{R} \,\frac{p'}{p}\left[|Z_2|^2-|Z_1|^2\right]+k\sqrt{R} \,\frac{p'}{p}\Re( B_L-1)|Z_1|^2.
\end{align}
Therefore, the energy function $E$ satisfies
the equation
\begin{align}\label{eq:energqueq}
	\ddt E=\sum_{i=1}^5\mathcal{I}_i.
\end{align}
where the error terms are defined as  
\begin{align}
	\mathcal{I}_1=&\ \frac{1}{4k\sqrt{R}} \left(\frac{ p'}{p^\frac12}\right)'\Re\left(Z_1\overline{Z_2}\right),\\
	\mathcal{I}_2=&\ \frac{k\sqrt{R}}{p^{\frac{1}{2}}} \Re((B_L-1)Z_1\overline{Z_2}),\\
	\mathcal{I}_3=&\ \frac14\frac{p'}{p}\Re( B_L-1)|Z_1|^2,\\
	\mathcal{I}_4=&\ -\beta\dfrac{k}{p}\Im(B_L) |Z_1|^2,\\
	\mathcal{I}_5=&\ \frac{\beta}{4\sqrt{R}} \frac{p'}{p^\frac32} \Re\left(i B_L Z_1\overline{Z_2}\right).
\end{align}
We now proceed to show that each $\mathcal{I}_i$ has good time-integrability, and therefore we can close the Gr\"onwall estimate which allow us to prove Theorem \ref{thm:Couette-main}. 

To control $\mathcal{I}_1$, notice that 
\begin{equation}
	\label{bd:dtdtp}
	 \left(\frac{p'}{p^{\frac12}}\right)' =\frac{2 k^2}{p^{\frac12}}-\frac12 \frac{(p')^2}{p^{\frac32}}=\frac{2k^4}{p^{\frac32}}= \frac{2|k|}{(1+|t-\frac{\eta}{k}|^2)^\frac32}.
\end{equation}
Therefore 
\begin{equation}
	\label{bd:I1strat}
	|\mathcal{I}_1|\leq \frac{1}{4\sqrt{R}} \frac{1}{(1+|t-\frac{\eta}{k}|^2)^\frac32}(|Z_1|^2+|Z_2|^2).
\end{equation}
Then, observe that the multiplier $B_L$ is contained in all the remaining terms. We can compute it explicitly from \eqref{eq:BLsymb}, namely, we know that $B_L^{-1}=1+i\beta(\eta-kt)p^{-1}=(p+i\beta(\eta-kt))p^{-1}$, hence
\begin{equation}
	B_L=\frac{p^2}{p^2+\beta^2(\eta-kt)^2}-i\beta p\frac{\eta-kt}{p^2+\beta^2(\eta-kt)^2}.
\end{equation}
In particular we have 
\begin{align}
	\label{bd:BL} |B_L|\leq& \ 1+\beta,\\
	\label{bd:ImBL}|\Im(B_L)|\leq& \ \frac{\beta}{p^{\frac12}},\\
	\label{bd:ReBL-1}|\Re(B_L-1)|=&\ \beta^2\frac{(\eta-kt)^2}{p^2+\beta^2(\eta-kt)^2}\leq \frac{\beta^2}{p},\\
	\label{bd:BL-1} |B_L-1|\leq & |\Re(B_L-1)|+|\Im(B_L)|\leq \frac{\beta+\beta^2}{p^{\frac12}}.
\end{align}
By using the bounds above and the fact that $|p'|\leq 2|k|p^{\frac12}$ we infer 
\begin{align}
	|\mathcal{I}_2|+|\mathcal{I}_5|\lesssim& \  \frac{|k|}{p}(|Z_1|^2+|Z_2|^2),\\
	|\mathcal{I}_3|+|\mathcal{I}_4|\lesssim&\    \frac{|k|}{p^\frac32}|Z_1|^2.
\end{align}
Combining the bounds above with \eqref{bd:I1strat}, we deduce that 
\begin{equation}
	|\sum_{i=1}^5 \mathcal{I}_i|\lesssim \frac{1}{1+(t-\frac{\eta}{k})^2}=\frac{k^2}{p},
\end{equation}
where we have roughly bounded all the terms with $k^2 p^{-1}$. Consequently, by the coercivity properties of $E$, see \eqref{eq:coercive-pointwise}, we get   
\begin{align}\label{eq:lowerbound}
	- \dfrac{1}{2 \sqrt{R}-1}\dfrac{|k|^2}{p}E \lesssim \ddt E \lesssim \dfrac{1}{2 \sqrt{R}-1}\dfrac{|k|^2}{p}E.
\end{align}
Now, since $\int_0^\infty |k|^2 p^{-1} \dd t \leq \pi/2$, we can apply the Gr\"onwall lemma to \eqref{eq:lowerbound} and obtain
\begin{align}
	E(t)\approx E(0).
\end{align}
This translates immediately into \eqref{eq:estimatecouette} thanks to \eqref{eq:coercive-pointwise}, thereby 
concluding the proof of Theorem \ref{thm:Couette-main}. \qed

\section{Shears close to Couette}
\label{sec:shearsClose}
This section deals with the full system \eqref{eq:EulerBmovek}. Using the expressions of $\Delta_t^{-1}$ and $B_t$ in \eqref{def:Deltat-1}-\eqref{def:TL} and \eqref{eq:NeumBt}-\eqref{def:TB}, the system reads
	\begin{align}
	\label{eq:pertofCouetteTheta}\dt \Theta_k&=-ik\RR Q_k+ik\left(b(Y)-\beta g(Y)\right) \Delta_t^{-1}B_t\Theta_k,\\
	\label{eq:pertofCouetteQ}\dt \QQ_k&=ik \Delta_L^{-1}\Theta_k + ik \Delta_L^{-1}(B_L-1)\Theta_k+ik \Delta_L^{-1}T_{\eps} T_L B_L\Theta_k+ik \Delta_t^{-1}B_\eps  B_t\Theta_k.
	\end{align}
We have expanded the right-hand side above to highlight  the similarities with the Couette flow  \eqref{eq:couettek}. In the first equation
\eqref{eq:pertofCouetteTheta}, the second term in the right-hand side is treated as an error term, just like in the previous section. In the
second equation \eqref{eq:pertofCouetteQ}, we have expanded the operator $\Delta_t^{-1}$ and $B_t$ in order to extract the Couette-like structure,
while all the other terms are treated as errors.
We have the following result for shear flows near Couette.
\begin{theorem}\label{thm:Close-main}
	Let $R>1/4$, $\beta \geq 0$ and  $s\geq 0$ be fixed. There exist $C_0>1$, $\eps_0\in (0,1/(2C_0))$ with the following property. If
	$\eps\in(0,\eps_0]$ and
	\begin{align}\label{eq:assgbclose}
	\| g-1\|_{s+5}+\| b\|_{s+4}\leq \eps,
	\end{align}
	then for every $k\neq 0$ the solution to \eqref{eq:pertofCouetteTheta}-\eqref{eq:pertofCouetteQ} satisfies the uniform $H^s$-bound
	\begin{align}\label{eq:estimateClosecouette}
	\norm{p^{-\frac14}\Theta_k(t)}_s^2+\norm{p^{\frac14}Q_k(t)}_s^2 \lesssim 
	\jap{t}^\delta\left(\norm{\Theta_k(0)}_{s}^2+\norm{Q_k(0)}_{s+1}^2\right), \qquad \forall t\geq 0,
	\end{align}
	where $\delta=C_0\eps$  and $p$ is given by \eqref{def:p}. 
\end{theorem}
The bound on the vorticity follows directly.
\begin{corollary}
	\label{cor:Omega}
	Under the same assumptions of Theorem \ref{thm:Close-main}, for $k\neq 0$ the following inequality holds true
	\begin{equation}
	\norm{p^{-\frac14}\Omega_k(t)}^2_s+\norm{p^{\frac14}Q_k(t)}^2_s \lesssim \ \jap{t}^\delta\left(\norm{\Omega_k(0)}_{s}^2+\norm{Q_k(0)}_{s+1}^2\right), \qquad \forall t\geq 0.
	\end{equation}
\end{corollary}
Appealing to Theorem \ref{thm:Close-main} and Corollary \ref{cor:Omega}, which will be proved throughout the paper, we first prove Theorem \ref{thm:CloseCouette}. 

\begin{proof}[Proof of Theorem \ref{thm:CloseCouette}]
	In order to prove \eqref{eq:main1}, first observe that since $$p^{-\frac14}\lesssim \jap{kt}^{-\frac12} \jap{k,\eta}^{\frac12}$$ we get
	\begin{align}
		\norm{q_k(t)}_{L^2}=\norm{Q_k(t)}_{L^2}=\norm{p^{-\frac14}p^{\frac14}Q_k(t)}_{L^2}\lesssim \frac{1}{\jap{kt}^{\frac12}}\norm{p^{\frac14}Q_k(t)}_{\frac12}
	\end{align}
	Hence, appealing to Corollary \ref{cor:Omega} and summing in $k$, we prove the bound for $q$ in \eqref{eq:main1}. For the velocity field, we know that its components in the coordinates driven by the flow read
	\begin{align}
		V^x=-g(Y)(\partial_Y-t\partial_X) \Delta_t^{-1} \Omega, \quad V^y=\partial_X \Delta_t^{-1} \Omega.
	\end{align}
	We begin with the bound on $V^x$. First rewrite $V^x$ as 
	\begin{align}
		V^x&=-(\de_Y-t\de_X)\Delta_t^{-1}\Omega-(g(Y)-1)(\de_Y-t\de_X)\Delta_t^{-1}\Omega\\
		&:= V^{x,1}+V^{x,\eps} 
	\end{align}
	Since $\Delta_t^{-1}=\Delta_L^{-1} T_L$ as in \eqref{def:Deltat-1}, we bound $V^{x,1}$ as follows
	\begin{align}
		\norm{V^{x,1}_k}_{L^2}^2=&\int \frac{(\eta-kt)^2}{p^2}|\widehat{T_L\Omega}_k|^2d\eta \leq \int \frac{1}{p}|\widehat{T_L\Omega}_k|^2d\eta\\
		\label{bd:Vx1}\lesssim& \frac{1}{\jap{kt}}\norm{p^{-\frac14}(T_L\Omega_k)(t)}_{\frac12}^2, 
	\end{align}
	where in the last inequality we have used again $p^{-\frac12}\lesssim \jap{kt}^{-1}\jap{k,\eta}$. Then, thanks to \eqref{def:TL} we know that $T_L=I+T_{\eps}T_L$, where $T_\eps$ is defined in \eqref{def:tildeTL}. Therefore 
	\begin{equation}
		\label{bd:TLOme}
		\norm{p^{-\frac14}T_L\Omega_k(t)}_{\frac12}\leq\norm{p^{-\frac14}\Omega_k(t)}_{\frac12}+\norm{p^{-\frac14}T_{\eps}T_L\Omega_k(t)}_{\frac12} 
	\end{equation}
	Now we would like to absorb the last term in the equation above in the left-hand side, similarly to what was done in the proof of Proposition \ref{prop-delta-t}. However, the multiplier $p^{-\frac14}$ does not commute with $T_\eps$. Hence, we have to use the inequality
	\begin{align}
		p^{-\frac{1}{4}}(t, k, \eta) \lesssim \l \eta-\xi \r^{\frac12} p^{-\frac{1}{4}}(t, k, \xi),
	\end{align}
	see also Lemma \ref{lem-exchange-freq-weights},	to exchange frequencies in the last term of the right-hand side of \eqref{bd:TLOme}. In particular, from the definition of $T_\eps$ in \eqref{def:tildeTL}, since $|\eta-kt|^2 p^{-1}\leq 1$, we deduce
	\begin{align*}
		\norm{p^{-\frac{1}{4}} T_{\eps}T_L \Omega_k}_{\frac{1}{2}}\leq&\norm{p^{-\frac{1}{4}}\left(|\widehat{g^2-1}|*|\widehat{T_L \Omega_k}|\right)}_{\frac12}+\norm{p^{-\frac{1}{4}}\left(|\widehat{b}|*|\widehat{T_L \Omega_k}|\right)}_{\frac12}\\
		\lesssim &\norm{\langle\cdot \rangle |\widehat{g^2-1}|*\langle \cdot \rangle^{\frac12} p^{-\frac{1}{4}}|\widehat{T_L \Omega_k}|}_{L^2}+\norm{\langle\cdot \rangle |\widehat{b}|*\langle\cdot \rangle^{\frac12}p^{-\frac{1}{4}}|\widehat{T_L \Omega_k}|}_{L^2}.
	\end{align*}
	Thanks to Young's convolution inequality, by combining the previous bound with \eqref{bd:TLOme} we infer
	\begin{align}
		\norm{p^{-\frac{1}{4}} T_L \Omega_k}_{\frac{1}{2}} \lesssim \norm{p^{-\frac{1}{4}} \Omega_k}_{\frac{1}{2}} + \eps \norm{p^{-\frac{1}{4}} T_L \Omega_k}_{\frac{1}{2}},
	\end{align}
	namely, for $\eps_0$ small enough 
	\begin{align}
		\label{bd:PTLOmega}
		\norm{p^{-\frac{1}{4}} T_L \Omega_k}_{\frac{1}{2}} \lesssim \norm{p^{-\frac{1}{4}}\Omega_k}_{\frac{1}{2}}.
	\end{align}
	Since $$\norm{V^{x,\eps}_k}_{L^2}\leq \norm{g-1}_{L^\infty}\norm{V^{x,1}_k}_{L^2}\lesssim \eps \norm{V^{x,1}_k}_{L^2},$$ by combining \eqref{bd:Vx1} with \eqref{bd:PTLOmega} and Corollary \ref{cor:Omega}, we get
	\begin{align}
		\norm{V^x_k}_{L^2}\lesssim \frac{1}{\jap{kt}^\frac12}\norm{p^{-\frac14}\Omega_k}_{\frac12} \lesssim \frac{1}{\jap{kt}^{\frac12-\frac{\delta}{2}}}\left(\norm{\Omega(0)_k}_{\frac12}+\norm{Q(0)_k}_{\frac32}\right).
	\end{align}
	Applying the same reasoning for $V^y_k$, we get
	\begin{align}\label{estimate-vClose}
		\|V_k^y(t)\|_{L^2} &\lesssim \frac{1}{\langle kt \rangle^{\frac{3}{2}-\frac{\delta}{2}}} \left(\|\Omega_k(0)\|_{\frac32}+\|Q_k(0)\|_{\frac52}\right).
	\end{align}
	Hence, by summing up in $k$ and defining $\delta_\eps=\delta/2=C_0\eps/2$, the proof of Theorem \ref{thm:CloseCouette} is concluded.
\end{proof}

\subsection{Symmetric variables and the energy functional}

Since \eqref{eq:pertofCouetteTheta}-\eqref{eq:pertofCouetteQ} decouples in the $X$-Fourier variable, we fix a nonzero integer $k$ and define the auxiliary variables 
\begin{equation}\label{eq:Z1Z2again}
Z_1:= m^{-1} p^{-\frac{1}{4}} \Theta_k, \qquad Z_2:=m^{-1} p^\frac{1}{4}i \sqrt{R} Q_k,
\end{equation}
with $m=m(t;k,\eta)$ a positive weight, that will be specified later, such that $m'>0$. The choice made in \eqref{eq:Z1Z2again}, up to the weight $m$ is exactly the one of the Couette case, see \eqref{eq:Z1Z2couette}. In fact, one immediately sees that $Z_1,Z_2$ satisfy
\begin{align}\label{eq:symm-matrix}
\de_t\begin{pmatrix}
Z_1\\
Z_2
\end{pmatrix}
=& \begin{pmatrix}
\displaystyle-\frac14\frac{p'}{p} & -k\sqrt{R}p^{-\frac12}\\
k\sqrt{R}p^{-\frac12} & \displaystyle\frac14\frac{p'}{p}
\end{pmatrix}
\begin{pmatrix}
Z_1\\
Z_2
\end{pmatrix}-\frac{m'}{m}\begin{pmatrix}
Z_1\\
Z_2
\end{pmatrix}+\mathcal{R}(t)\begin{pmatrix}
Z_1\\
Z_2
\end{pmatrix},
\end{align}
where $\mathcal{R}(t)$ will be treated as a perturbation being indeed a matrix of remainders. The error terms which are hidden (up to now) in $\mathcal{R}(t)$ need to be controlled by means of the artificial dissipation introduced by the weight $m$.

Given a real number $s\geq 0$, we define the energy functional as
\begin{align}\label{def:energyfun}
E_s(t)=\frac12\left[\|Z_1(t)\|_s^2+\|Z_2(t)\|_s^2+\frac{1}{2k\sqrt{R}} \Re \l p' p^{-\frac12} Z_1(t),Z_2(t)\r_s\right],
\end{align} 
which is exactly the one used for the Couette case, see \eqref{def:pointwise-functional-Couette}, when integrated in the variable $\eta$.

Now using that $|p'|p^{-\frac12}\leq 2|k|$, we deduce that
\begin{align}\label{eq:estimate-mixed-term}
\frac{1}{2|k|\sqrt{R}} \left| \l p' p^{-\frac12} Z_1,Z_2\r_s \right|\leq\frac{1}{\sqrt{R}}\|Z_1\|_s\|Z_2\|_s\leq\frac{1}{2\sqrt{R}}\left(\|Z_1\|_s^2+\|Z_2\|_s^2\right).
\end{align}
As a consequence,  the functional is coercive whenever $R>1/4$, namely
\begin{align}\label{eq:coercive}
\frac12\left(1-\frac{1}{2\sqrt{R}}\right)\left[\|Z_1\|_s^2+\|Z_2\|_s^2\right]\leq E_s\leq\frac12\left(1+\frac{1}{2\sqrt{R}}\right)\left[\|Z_1\|_s^2+\|Z_2\|_s^2\right].
\end{align}The rest of the paper aims at proving that $t\to E_s(t)$ is non-increasing. 
\begin{lemma}
	The functional $E_s(t)$ satisfies 
	\begin{equation}\label{dtEs1}
	\ddt E_s+\left(1-\frac{1}{2\sqrt{R}}\right)\left[\norm{\sqrt{\frac{m'}{m}}Z_1}_s^2+\norm{\sqrt{\frac{m'}{m}}Z_2}_s^2\right]
	\leq\frac{1}{4|k|\sqrt{R}}\left| \l\left(p' p^{-\frac12}\right)'Z_1,Z_2\r_s \right|+\sum_{i=1}^8\mathcal{R}_i,
	\end{equation}
	where the error terms $\mathcal{R}_i$ are 
	\begin{align}
	\label{def:R1}\mathcal{R}_1&=|k|\left|\l Z_1, m^{-1} p^{-\frac{1}{4}}\left((b-\beta g) \Delta_L^{-1}T_L T_B B_L\Theta\right) \r_s\right|,\\
	\label{def:R2}\mathcal{R}_2&=|k|\sqrt{R}\, \left|\l p^{-\frac{1}{2}}(B_L-1)Z_1, Z_2 \r_s\right|,\\
	\label{def:R3}\mathcal{R}_3&=|k|\sqrt{R}\, \left| \l m^{-1}p^{-\frac{3}{4}}T_{\eps}T_LB_L\Theta, Z_2 \r_s\right|,\\
	\label{def:R4}\mathcal{R}_4&=|k|\sqrt{R}\, \left| \l m^{-1}p^{-\frac{3}{4}}T_LB_\eps T_B B_L\Theta, Z_2 \r_s\right|,\\
	\label{def:R5} \mathcal{R}_5&=\frac{1}{4\sqrt{R}}\left|\l p' p^{-\frac34} m^{-1} \left((b-\beta g)\Delta_L^{-1} T_L T_B B_L\Theta\right),Z_2\r_s\right|,\\
	\label{def:R6}\mathcal{R}_6&=\frac14\left|\l Z_1, \frac{p'}{p} (B_L-1)Z_1\r_s\right|,\\
	\label{def:R7}\mathcal{R}_7&=\frac14\left| \l \frac{p'}{p} Z_1, m^{-1}p^{-\frac14} T_{\eps}T_LB_L\Theta\r_s\right|,\\
	\label{def:R8}\mathcal{R}_8&=\frac14\left| \l  \frac{p'}{p} Z_1, m^{-1}p^{-\frac14} T_LB_{\eps}T_BB_L\Theta\r_s\right|.
	\end{align}
\end{lemma}
\begin{proof}
	Taking the time-derivative of the functional, from $Z_1$ we have
	\begin{align}
	\label{eq:Z1close1}\frac12\ddt\|Z_1\|^2_s&=-\norm{\sqrt{\frac{m'}{m}}Z_1}^2_s-\frac14 \l Z_1, \frac{p'}{p}Z_1 \r-k\sqrt{R}\,\Re\l Z_1,p^{-\frac{1}{2}}Z_2 \r_s\\
	& \label{Z1close2}\quad +\Re\l Z_1, ik m^{-1} p^{-\frac{1}{4}}\left((b-\beta g) \Delta_L^{-1}T_LT_B B_L\Theta\right) \r_s.
	\end{align}
	About $Z_2$,
	\begin{align}\label{eq:Z2close}
	\frac12\ddt\|Z_2\|_s^2&=-\norm{\sqrt{\frac{m'}{m}}Z_2}_s^2+\frac14 \l Z_2, \frac{p'}{p}Z_2 \r_s+k\sqrt{R}\, \Re \l p^{-\frac12}Z_1, Z_2 \r_s\\
	&\quad+k\sqrt{R}\, \Re \l p^{-\frac{1}{2}}(B_L-1)Z_1, Z_2 \r_s +k\sqrt{R}\, \Re \l m^{-1}p^{-\frac{3}{4}}T_{\eps}T_LB_L\Theta, Z_2 \r_s\notag\\
	&\quad+k\sqrt{R}\, \Re \l m^{-1}p^{-\frac{3}{4}}T_LB_\eps T_B B_L\Theta, Z_2 \r_s.
	\end{align}	
	For the mixed term we have
	\begin{align}\label{eq:mixedclose}
	\ddt \Re \l p' p^{-\frac12} Z_1,Z_2\r_s
	&=\Re \l \left(p' p^{-\frac12}\right)'Z_1,Z_2\r_s -2\l p' p^{-\frac12}  m'  m^{-1} Z_1,Z_2\r_s -k\sqrt{R}\, \l Z_2, \frac{p'}{p}Z_2 \r_s\\
	&\quad +\Re \l ikp' p^{-\frac34} m^{-1} \left((b-\beta g)\Delta_L^{-1}T_LB_t\Theta\right),Z_2\r_s\\
	\label{eq:canc}&\quad+k\sqrt{R}\,  \l Z_1, \frac{p'}{p} Z_1\r_s\\
	&\quad+k\sqrt{R}\,  \Re \l Z_1, \frac{p'}{p} (B_L-1)Z_1\r_s+k\sqrt{R}\, \Re \l \frac{p'}{p} Z_1, m^{-1}p^{-\frac14} T_{\eps}T_LB_L\Theta\r_s\\
	&\quad+k\sqrt{R}\, \Re \l  \frac{p'}{p} Z_1, m^{-1}p^{-\frac14} T_LB_{\eps}T_BB_L\Theta\r_s.
	\end{align}
	Now, notice that the sum of the last term on the right-hand side of \eqref{eq:Z1close1} and the last term on the right-hand side of \eqref{eq:Z2close} is zero. Next, the last term of \eqref{eq:mixedclose} multiplied by $1/(4k\sqrt{R})$ and the second term in the right-hand side of \eqref{eq:Z2close} balance each other, while the sum of  \eqref{eq:canc} multiplied by $1/(4k\sqrt{R})$ and the second term in the right-hand side of \eqref{eq:Z1close1} vanishes.
	Thus, we end up with
	
	\begin{align}
	\ddt E_s(t) & + \norma{\sqrt{\dfrac{m'}{m}} Z_1}{s}^2 + \norma{\sqrt{\dfrac{m'}{m}} Z_2}{s}^2\notag\\
	&\le \dfrac{1}{4|k| \sqrt{R}} \left| \l \left(p' p^{-\frac12}\right)'Z_1,Z_2\r_s\right| + \dfrac{1}{2|k| \sqrt{R}}\left|\l p' p^{-\frac12}  m'  m^{-1} Z_1,Z_2\r_s\right|+\sum_{i=1}^8 \mathcal{R}_i.
	\end{align}	
	This way, using that $|p'p^{-\frac12}|\leq 2|k|$,
	we get
	\begin{align}\label{eq:easyerr}
	\frac{1}{2|k|\sqrt{R}}\left|\l p' p^{-\frac12} m' m^{-1} Z_1,Z_2\r_s\right|
	\leq\frac{1}{2\sqrt{R}}\, \left(\norm{\sqrt{\frac{m'}{m}}Z_1}_s^2+\norm{\sqrt{\frac{m'}{m}}Z_2}_s^2\right),
	\end{align}
	and the proof is concluded.
\end{proof}

\subsection{Choice of weights and their properties}
The crucial observation for a proper choice of the weights is that 
we need an additional (uniformly bounded) weight to control the error terms generated by the various $\mathcal{R}_i$'s. More precisely, as it will be clarified later on, the time-derivative of this new weight has to be comparable with $|p'|/p$. This implies that the decay rates proven in Theorem \ref{thm:CloseCouette} have a small $\delta$-correction compared to those in Couette.
For these reasons, the weight $m$ in the definition of $Z_1$ and $Z_2$ in \eqref{eq:Z1Z2again} now takes the form
\begin{align}
\label{def:mSc}
m=m_1 w^\delta,
\end{align} 
where we recall that $\delta=C_0\eps $. The weight $w$ encodes the decay correction, so that
\begin{equation}\label{def:w}
\dfrac{w'}{w}=\dfrac{1}{4}\dfrac{|p'|}{p}, \qquad w|_{t=0}=1.
\end{equation}
The choice of the weight $m_1$,
\begin{align}\label{def:m1}
\frac{m_1'}{m_1}=C_\beta\dfrac{|k|^2}{p}, \qquad m_1|_{t=0}=1,
\end{align}
where the constant $C_\beta$ is given by 
\begin{equation}
\label{def:Cbeta}
C_\beta = 256\sqrt{R}\left(\frac{2\sqrt{R}}{2\sqrt{R}-1}\right) (1+\beta^2),
\end{equation}
plays a role in controlling the right-hand side of \eqref{dtEs1}. It is worth pointing out that $C_\beta$ blows up as $R\to 1/4$. 
Notice that
\begin{align}\label{eq:sumofweight}
\frac{m'}{m}=\delta\frac{w'}{w}+\frac{m_1'}{m_1}.
\end{align}
Explicitly, for $w$ we have
\begin{align}\label{expression:w}
w(t;k,\eta)=\begin{cases}
\left(\dfrac{k^2+\eta^2}{p(t;k,\eta)}\right)^\frac14, \qquad &t < \dfrac{\eta}{k}, \\
\left(\dfrac{(k^2+\eta^2)p(t;k,\eta)}{k^4}\right)^\frac14, \qquad &t \ge \dfrac{\eta}{k},
\end{cases}
\end{align} 
while 
\begin{align}\label{expression:m_1}
m_1(t;k,\eta)=\exp \left[C_\beta\left(\arctan\left(\frac{\eta}{k}-t\right)-\arctan\left(\frac{\eta}{k}\right)\right)\right].
\end{align}
In particular, the weight $m_1$ is uniformly bounded. 

\begin{remark}
	The weight $m_1$ is standard in the case of incompressible Euler/Navier-Stokes with constant density \cites{BVW16,Zillinger14,Zillinger15}.
\end{remark}

The goal of the remaining part of this section is to prove the following proposition.
\begin{proposition}\label{prop:Esfinal}
	Let $R>1/4$, $\beta \geq 0 $ and  $s\geq 0$ be fixed. There exist $\eps_0\in (0,1)$ with the following property. If
	$\eps\in(0,\eps_0]$  and
	\begin{align}
	\| g-1\|_{s+5}+\| b\|_{s+4}\leq \eps,
	\end{align}
	then
	\begin{align}
	\ddt E_s+\frac14\left(1-\frac{1}{2\sqrt{R}}\right)\left[\norm{\sqrt{\frac{m'}{m}}Z_1}_s^2+\norm{\sqrt{\frac{m'}{m}}Z_2}^2_s\right]\leq0,
	\end{align}
	for every $t\geq 0$. 
\end{proposition}
Theorem \ref{thm:Close-main} follows from the above proposition (see Section \ref{sub:proofs}).
The proof of this proposition is a consequence of the properties of the chosen weights (see Lemmas and \ref{lem-exchange-freq-weights} and
\ref{lem-exchange-m1-w} below), and the estimates on the various error terms, postponed in the next Section \ref{sub:erroest}. We begin
by computing how much it costs to exchange weights in the various convolutions appearing in the error terms.

\begin{lemma}\label{lem-exchange-freq-weights}
	Let $k\neq0$ and $t\geq 0$ be fixed. For any $\eta,\xi\in\R$ we have 
	\begin{align}
	\label{exchange-p} & p^{-1}(t;k,\eta) \lesssim \langle \eta-\xi \rangle^2 p^{-1}(t;k,\xi),\\
	\label{exchange-p'} & \dfrac{|p'|}{p}(t;k,\eta) \lesssim  \langle {{\eta-\xi}}\rangle^2 \dfrac{|p'|}{p}(t;k,\eta)+|k| \langle {\eta-\xi} \rangle^3 p^{-1}(t;k,\eta),\\
	\label{exchange-m} & m^{-1}(t;k,\eta) \lesssim  \langle \eta-\xi \rangle^{\delta} m^{-1}(t; k, \xi).
	\end{align}
\end{lemma}

\begin{proof}
	In the course of the proof, we omit the dependency of all the quantities on $k$ and $t$.
	Starting with \eqref{exchange-p}, the inequality to prove
	is equivalent to
	
	\begin{align}
	\langle \dfrac{\xi}{k}-t \rangle^2 \lesssim \langle \dfrac{\eta-\xi}{k}\rangle^2\langle \dfrac{\eta}{k}-t \rangle^2.
	\end{align}
	Choosing $a=\dfrac{\xi}{k}-t$ and $b=\dfrac{\eta}{k}-t$, this follows from the general inequality 
	
	\begin{align}\label{eq-genenal-ineq-ab}
	\langle a \rangle \lesssim  \langle a-b \rangle \langle b \rangle.
	\end{align}
	Turning to \eqref{exchange-p'}, we write
	\begin{align}
	\dfrac{|p'|}{p}(\eta) & = \dfrac{2|\dfrac{\eta}{k}-t|}{ \langle \dfrac{\eta}{k}-t \rangle^2} \le \dfrac{2|\dfrac{\xi}{k}-t|+2|\dfrac{\xi}{k}-\dfrac{\eta}{k}|}{\langle \dfrac{\eta}{k}-t\rangle^2}\\
	& \le \left[2|\dfrac{\xi}{k}-t|+2|\dfrac{\xi}{k}-\dfrac{\eta}{k}|\right] \langle\dfrac{\eta}{k}-\dfrac{\xi}{k}\rangle^2 \dfrac{1}{\langle \dfrac{\xi}{k}-t\rangle^2}\\
	& \le  \left[2|\dfrac{\xi}{k}-t|+2|\dfrac{\xi}{k}-\dfrac{\eta}{k}|\right] \langle\eta-\xi\rangle^2 k^2 p^{-1}(\xi)\\
	& \lesssim \dfrac{|p'|}{p}(\xi) \langle {\eta-\xi}\rangle^2+|k| \langle \eta-\xi \rangle^3 p^{-1}(\xi).
	\end{align}
	We now prove \eqref{exchange-m}. Recalling that
	
	\begin{align}
	m^{-1}=(w^\delta m_1)^{-1}, 
	\end{align}
	where $w, m_1$ are is by \eqref{expression:w}, \eqref{expression:m_1} and, since $m_1$ is uniformly bounded, then
	\begin{align}
	(m_1)^{-1}(\eta) \lesssim (m_1)^{-1}(\xi).
	\end{align}
	Thus we deal with $w^{-\delta}$. We know from \eqref{exchange-p} that
	\begin{align}
	p^{-1}(\xi) \lesssim \langle \eta-\xi \rangle^2 p^{-1}(\eta).
	\end{align} 
	Therefore
	\begin{align}
	p(\eta) \lesssim \langle \eta-\xi \rangle^2 p(\xi).
	\end{align}
	Notice that 
	
	\begin{align}
	\dfrac{p(\eta)}{k^2+\eta^2} \lesssim  \langle \eta-\xi \rangle^4 \dfrac{p(\xi)}{k^2+\xi^2}.
	\end{align}
	Recalling the expression of $w$ in \eqref{expression:w} for $t < \dfrac{\eta}{k}$, it follows that
	\begin{align}
	w^{-\delta}(\eta) \le \langle \eta-\xi \rangle^{\delta} w^{-\delta}(\xi).
	\end{align}
	The complementing case $t \ge \dfrac{\eta}{k}$ is analogous, and the proof is over.
\end{proof}
\begin{remark}l\label{rmk-exp-constant}
	It is worth pointing out that the constant value which is implicitly involved in the right-hand side of \eqref{exchange-m} (through the notation `` $\lesssim$ '' ) is exactly $\e^{2C_\beta}$, up to additional constants (independent of $\eps, \beta$). This exponential term $\e^{2C_\beta}$ comes directly from the explicit expression of the weight $m_1$ in \eqref{expression:m_1}.
\end{remark}
We also need to compute how much it costs to commute $T_L, T_\eps, T_B, B_\eps$ with the weights.
\begin{lemma}\label{lem-exchange-m1-w} Let $R>1/4$, $\beta \geq 0$ and  $s\geq 0$ be fixed. There exist $\eps_0\in (0,1)$ with the following property. If
	$\eps\in(0,\eps_0]$ and
	\begin{align}\label{eq:assgbclose2}
	\| g-1\|_{s+5}+\| b\|_{s+4}\leq \eps, 
	\end{align}
	then, for every smooth function $f$ and every $t\geq 0$, the following estimates hold true
	\begin{align}\label{eq:exchangem1}
	\norm{\sqrt{\dfrac{m_1'}{m_1}} p^{-\frac14}m^{-1}T_{\eps}f}_{s}\lesssim& \ \eps \norm{ \sqrt{\dfrac{m_1'}{m_1}} p^{-\frac14}m^{-1} f }_{s},\\
	\label{bd:exchangem2}\norm{\sqrt{\dfrac{m_1'}{m_1}} p^{-\frac14}m^{-1}B_\eps f}_{s}\lesssim&\ \eps  \norm{\sqrt{\dfrac{m_1'}{m_1}} p^{-\frac14}m^{-1}f}_{s},
	\end{align}
	and
	\begin{align}\label{eq:exchangew}
	\norm{\sqrt{\dfrac{w'}{w}} p^{-\frac14}m^{-1}T_{\eps}f}_{s}\lesssim&    \ \eps\left(\norm{\sqrt{\dfrac{w'}{w}} p^{-\frac14}m^{-1}f}_{s}+\norm{\sqrt{\dfrac{m_1'}{m_1}} p^{-\frac14}m^{-1} f}_{s}\right),
	\\
	\label{bd:wtildeBt}\norm{\sqrt{\dfrac{w'}{w}} p^{-\frac14}m^{-1}B_\eps f}_{s}\lesssim&\ \eps \left(\norm{\sqrt{\dfrac{w'}{w}} p^{-\frac14}m^{-1}f}_{s}+\norm{\sqrt{\dfrac{m_1'}{m_1}} p^{-\frac14}m^{-1} f}_{s}\right). 
	\end{align}
	In addition, the following inequalities hold 
	\begin{align}
	\label{bd:m1TL}\norm{\sqrt{\dfrac{m_1'}{m_1}} p^{-\frac14}m^{-1}T_Lf}_{s}\leq& \   2\norm{ \sqrt{\dfrac{m_1'}{m_1}} p^{-\frac14}m^{-1} f }_{s},\\
	\label{bd:m1Bt}\norm{\sqrt{\dfrac{m_1'}{m_1}} p^{-\frac14}m^{-1}T_B f}_{s}\leq& \ 2  \norm{\sqrt{\dfrac{m_1'}{m_1}} p^{-\frac14}m^{-1}f}_{s},
	\end{align}
	and
	\begin{align}\label{bd:wTL}
	\norm{\sqrt{\dfrac{w'}{w}} p^{-\frac14}m^{-1}T_Lf}_{s}\leq& \ 2  \left( \norm{\sqrt{\dfrac{w'}{w}} p^{-\frac14}m^{-1}f}_{s}+\norm{\sqrt{\dfrac{m_1'}{m_1}} p^{-\frac14}m^{-1} f}_{s}\right),\\
	\label{bd:wBt}\norm{\sqrt{\dfrac{w'}{w}} p^{-\frac14}m^{-1}T_B f}_{s}\leq&  \   2\left( \norm{\sqrt{\dfrac{w'}{w}} p^{-\frac14}m^{-1}f}_{s}+\norm{\sqrt{\dfrac{m_1'}{m_1}} p^{-\frac14}m^{-1} f}_{s}\right). 
	\end{align}
\end{lemma}

\begin{proof} We start with $T_L, T_\eps$.	In order to prove \eqref{eq:exchangem1}, we first use the operators $T_{\eps}^g$ and $T_{\eps}^b$ defined in \eqref{eq:Tgandb} and 
	the identity 
	\begin{align}
	T_{\eps}=T_{\eps}^g+T_{\eps}^b
	\end{align} 
	to obtain the bound
	\begin{align}\label{eq:estm1m1}
	\norm{\sqrt{\dfrac{m_1'}{m_1}} p^{-\frac14}m^{-1}T_{\eps}f}_{s}\leq 
	\norm{\sqrt{\dfrac{m_1'}{m_1}} p^{-\frac14}m^{-1}T_{\eps}^gf}_{s}
	+\norm{\sqrt{\dfrac{m_1'}{m_1}} p^{-\frac14}m^{-1}T_{\eps}^bf}_{s}.
	\end{align}
	Now, by Lemma \ref{lem-exchange-freq-weights} we find that
	\begin{align}
	\label{bd:commm1}
	\left(\sqrt{\dfrac{m_1'}{m_1}} p^{-\frac14}m^{-1}\right)(\eta)\lesssim \l\xi-\eta\r^{\frac32+\delta}\left(\sqrt{\dfrac{m_1'}{m_1}} p^{-\frac14}m^{-1}\right)(\xi),
	\end{align}
	where we recall that the hidden constant in the right-hand side of the previous bound \eqref{bd:commm1}  is proportional to $\e^{2C_\beta}$, see Remark \ref{rmk-exp-constant}.
	Therefore,
	\begin{align}
	\norm{\sqrt{\dfrac{m_1'}{m_1}} p^{-\frac14}m^{-1}T_{\eps}^gf}_{s}
	&\leq\norm{\l(k,\cdot)\r^s\sqrt{\dfrac{m_1'}{m_1}} p^{-\frac14}m^{-1}|\widehat{g^2-1}|*|\widehat{f}|}\notag\\
	&\lesssim\norm{\l\cdot\r^{s+\frac32+\delta}|\widehat{g^2-1}|*\l(k,\cdot)\r^s\sqrt{\dfrac{m_1'}{m_1}} p^{-\frac14}m^{-1}|\widehat{f}|}\notag\\
	\label{bd:Ttildeg}&\lesssim \norm{g^2-1}_{s+\frac52+\delta} \norm{\sqrt{\dfrac{m_1'}{m_1}} p^{-\frac14}m^{-1}f}_s.
	\end{align}
	In a similar fashion,
	\begin{align}
	\norm{\sqrt{\dfrac{m_1'}{m_1}} p^{-\frac14}m^{-1}T_{\eps}^bf}_{s}
	&\lesssim \norm{b}_{s+\frac52+\delta}\norm{\sqrt{\dfrac{m_1'}{m_1}} p^{-\frac14}m^{-1}f}_s.
	\end{align}
	Plugging the above two bounds in \eqref{eq:estm1m1} and using \eqref{eq:assgbclose2} we obtain \eqref{eq:exchangem1}.
	Looking at \eqref{eq:exchangew}, we have again that
	\begin{align}
	\norm{\sqrt{\dfrac{w'}{w}} p^{-\frac14}m^{-1}T_{\eps}f}_{s}\leq\norm{\sqrt{\dfrac{w'}{w}} p^{-\frac14}m^{-1}T_{\eps}^gf}_{s}+\norm{\sqrt{\dfrac{w'}{w}} p^{-\frac14}m^{-1}T_{\eps}^bf}_{s}.
	\end{align}
	As above, Lemma \ref{lem-exchange-freq-weights} now gives us
	\begin{align}
	\left(\sqrt{\dfrac{w'}{w}} p^{-\frac14}m^{-1}\right)(\eta)\lesssim \l {{\eta-\xi}}\r^{\frac32+\delta}\left(\sqrt{\dfrac{w'}{w}} p^{-\frac14}m^{-1}\right)(\xi)+ \langle {\eta-\xi} \rangle^{2+\delta} \left(\sqrt{\dfrac{m_1'}{m_1}} p^{-\frac14}m^{-1}\right)(\xi).
	\end{align}
	Therefore,
	\begin{align}
	\norm{\sqrt{\dfrac{w'}{w}} p^{-\frac14}m^{-1}T_{\eps}^gf}_{s}
	&\leq\norm{\l(k,\cdot)\r^s\sqrt{\dfrac{w'}{w}} p^{-\frac14}m^{-1}|\widehat{g^2-1}|*|\widehat{f}|}\notag\\
	&\lesssim\norm{\l\cdot\r^{s+\frac32+\delta}|\widehat{g^2-1}|*\l(k,\cdot)\r^s\sqrt{\dfrac{w'}{w}} p^{-\frac14}m^{-1}|\widehat{f}|}\notag\\
	&\quad+\norm{\l\cdot\r^{s+2+\delta}|\widehat{g^2-1}|*\l(k,\cdot)\r^s\sqrt{\dfrac{m_1'}{m_1}} p^{-\frac14}m^{-1}|\widehat{f}|}\notag\\
	&\lesssim \norm{g^2-1}_{s+3+\delta} \left[\norm{\sqrt{\dfrac{w'}{w}} p^{-\frac14}m^{-1}f}_s+\norm{\sqrt{\dfrac{m_1'}{m_1}} p^{-\frac14}m^{-1}f}_s\right].
	\end{align}
	The part with $T_{\eps}^b$ is similar, so that
	\begin{align}
	\norm{\sqrt{\dfrac{w'}{w}} p^{-\frac14}m^{-1}T_{\eps}^bf}_{s}
	&\lesssim \norm{b}_{s+3+\delta} \left[\norm{\sqrt{\dfrac{w'}{w}} p^{-\frac14}m^{-1}f}_s+\norm{\sqrt{\dfrac{m_1'}{m_1}} p^{-\frac14}m^{-1}f}_s\right].
	\end{align}
	By combining the previous two bounds we prove \eqref{eq:exchangew}. 
	About \eqref{bd:m1TL}, since $T_L=I+T_{\eps}T_L$, see \eqref{def:TL}, by using \eqref{eq:exchangem1} we get 
	\begin{align}
	\norm{\sqrt{\dfrac{m_1'}{m_1}} p^{-\frac14}m^{-1}T_Lf}_{s}\leq \  \norm{ \sqrt{\dfrac{m_1'}{m_1}} p^{-\frac14}m^{-1} f }_{s}+C\eps_0  \norm{\sqrt{\dfrac{m_1'}{m_1}} p^{-\frac14}m^{-1}T_Lf}_{s},
	\end{align}
	where $C$ is properly chosen so that the constant value which is hidden in the right-hand side of \eqref{eq:exchangem1} is compensated. Therefore, for $\eps_0$ small enough
	we prove \eqref{bd:m1TL}.  
	To deal with \eqref{bd:wTL}, we combine \eqref{def:TL} with \eqref{eq:exchangew} to get
	\begin{align}
	\norm{\sqrt{\dfrac{w'}{w}} p^{-\frac14}m^{-1}T_Lf}_{s}& \leq \  \norm{ \sqrt{\dfrac{w'}{w}} p^{-\frac14}m^{-1} f }_{s}
	+C\eps_0\norm{\sqrt{\dfrac{m_1'}{m_1}} p^{-\frac14}m^{-1}T_Lf}_{s}
	\\
	& \quad +C\eps_0\norm{\sqrt{\dfrac{w'}{w}} p^{-\frac14}m^{-1}T_Lf}_{s}.
	\end{align}
	For $\eps_0$ small, we can absorb the last term using the left-hand side. Applying \eqref{bd:m1TL} we end up with \eqref{bd:wTL}. 
	
	Now we deal with $B_\eps$ and $T_B$, starting with  \eqref{bd:exchangem2}. From the definition of $B_\eps$  in \eqref{def:Beps}, since $|B_L|\leq1$ as in \eqref{eq:BLboundbelow}, we get 
	\begin{align}
	\label{bd:Btilde1}\norm{\sqrt{\dfrac{m_1'}{m_1}} p^{-\frac14}m^{-1}B_\eps f}_{s}\lesssim& \norm{\sqrt{\dfrac{m_1'}{m_1}} p^{-\frac14}m^{-1}\left(\widehat{g-1}*\frac{p'}{2k}p^{-1}\widehat{f}\right)}_{s}\\
	\label{bd:Btilde2}&+\norm{\sqrt{\dfrac{m_1'}{m_1}} p^{-\frac14}m^{-1}\frac{p'}{2k}p^{-1}\widehat{T_\eps T_Lf}}_{s}\\
	\label{bd:Btilde3}&+  \norm{\sqrt{\dfrac{m_1'}{m_1}} p^{-\frac14}m^{-1}\left(\widehat{g-1}*\frac{p'}{2k}p^{-1}\widehat{T_\eps T_Lf}\right)}_{s}.
	\end{align}
	Considering the term in the right-hand side of \eqref{bd:Btilde1}, by using \eqref{bd:commm1} we get
	\begin{align}
	 \norm{\sqrt{\dfrac{m_1'}{m_1}} p^{-\frac14}m^{-1}\left(\widehat{g-1}*\frac{p'}{2k}p^{-1}\widehat{f}\right)}_{s}& \lesssim   \norm{\langle \cdot \rangle^{s+\frac32+\delta}|\widehat{g-1}|*\langle(k,\cdot)\rangle \sqrt{\dfrac{m_1'}{m_1}} p^{-\frac14}m^{-1}\widehat{f}}\\
	\label{bd:betag-1} &\lesssim \norm{g-1}_{s+\frac52+\delta}\norm{\sqrt{\dfrac{m_1'}{m_1}} p^{-\frac14}m^{-1}f}_{s}.
	\end{align}
	Thanks to \eqref{eq:exchangem1} and \eqref{bd:m1TL} we obtain 
	\begin{align}
	\label{bd:Btilde2bd}
	\norm{\sqrt{\dfrac{m_1'}{m_1}} p^{-\frac14}m^{-1}\frac{p'}{2k}p^{-1}\widehat{T_\eps T_Lf}}_{s}\lesssim	\ \eps\norm{\sqrt{\dfrac{m_1'}{m_1}} p^{-\frac14}m^{-1}f}_{s}.
	\end{align}
	Proceeding as done to get \eqref{bd:betag-1}, by using the previous bound we also have 
	\begin{equation}
		\label{bd:Btilde3bd}
	\norm{\sqrt{\dfrac{m_1'}{m_1}} p^{-\frac14}m^{-1}\left(\widehat{g-1}*\frac{p'}{2k}p^{-1}\widehat{T_\eps T_Lf}\right)}_{s}\lesssim	\ \eps\norm{\sqrt{\dfrac{m_1'}{m_1}} p^{-\frac14}m^{-1}f}_{s}.
	\end{equation}
	Therefore, combining \eqref{bd:Btilde1} with \eqref{bd:betag-1}, \eqref{bd:Btilde2bd} and \eqref{bd:Btilde3bd} we have that \eqref{bd:exchangem2} holds true.  
	
From \eqref{def:Beps} and the upper bound on $B_L$ in \eqref{eq:BLboundbelow}, it is easy to check that \eqref{bd:wtildeBt} follows by arguing as already done to obtain \eqref{eq:exchangew}. Turning to \eqref{bd:m1Bt}, by using \eqref{eq:NeumBt} we have 
	\begin{equation}
\norm{\sqrt{\dfrac{m_1'}{m_1}} p^{-\frac14}m^{-1}T_B f}_{s}\leq \norm{\sqrt{\dfrac{m_1'}{m_1}} p^{-\frac14}m^{-1}f}_{s}+\norm{\sqrt{\dfrac{m_1'}{m_1}} p^{-\frac14}m^{-1}B_\eps T_Bf}_{s}, 
	\end{equation}
	therefore, since $|B_L|\leq 1$, by using \eqref{bd:exchangem2} we get
	\begin{equation}
	\norm{\sqrt{\dfrac{m_1'}{m_1}} p^{-\frac14}m^{-1}T_Bf}_{s}\leq \norm{\sqrt{\dfrac{m_1'}{m_1}} p^{-\frac14}m^{-1}f}_{s}+C\eps_0\norm{\sqrt{\dfrac{m_1'}{m_1}} p^{-\frac14}m^{-1} T_B f}_{s},
	\end{equation}
	which proves \eqref{bd:m1Bt} choosing $\eps_0$ small enough. The proof of \eqref{bd:wBt} is similar.
\end{proof}

\subsection{Proof of Proposition \ref{prop:Esfinal}}\label{sub:erroest} 
The starting point is the differential inequality \eqref{dtEs1}. 
Exploiting the identity \eqref{bd:dtdtp}, we have that $( p'p^{-\frac{1}{2}})'\lesssim C_\beta^{-1}|k| m'_1/m_1$, hence we get
\begin{align}\label{eq:errm1Close}
\frac{1}{4|k|\sqrt{R}}\left| \l\left(p' p^{-\frac12}\right)'Z_1,Z_2\r_s \right|
\leq \frac{1}{2C_\beta}\left(\norm{\sqrt{\frac{m_1'}{m_1}}Z_1}_s^2+\norm{\sqrt{\frac{m_1'}{m_1}}Z_2}_s^2\right),
\end{align}
where we recall $C_\beta$ is given in \eqref{def:Cbeta}.

It remains to deal with the  $\mathcal{R}_i$'s. The idea is to absorb the estimates of these terms by using the second (positive) term in the left-hand side of \eqref{dtEs1}. These computations are developed below.

\medskip

\noindent\textbf{Estimate on  $\mathcal{R}_1$.}
We want to show that
\begin{align}\label{bd:R1est}
\mathcal{R}_1=|k|\left|\l Z_1, m^{-1} p^{-\frac{1}{4}}\left((b-\beta g) \Delta_L^{-1}T_L T_B B_L\Theta\right) \r_s\right|\leq \frac{1}{32}\left(1-\frac{1}{2\sqrt{R}}\right)\norma{\sqrt{\frac{m'_1}{m_1}} Z_1}{s}^2.
\end{align}
By using that $m_1'/m_1=C_\beta |k|^2/p$ we have
\begin{align}
\label{bd:form1}\mathcal{R}_1\leq& \l |Z_1|, m^{-1}p^{-\frac14}\left((|\widehat{b}|+\beta |\widehat{g-1}|)*|k|p^{-1}|\widehat{T_L T_B B_L\Theta}|\right)\r_s\\
\label{bd:form2}&+\beta \l |Z_1|, m^{-1}p^{-\frac14}|k|p^{-1}|\widehat{T_L T_B B_L\Theta}|\r_s\\
\label{bd:R10}\leq & \dfrac{1}{C_\beta|k|} \l |Z_1|, m^{-1}p^{-\frac14}\left((|\widehat{b}|+\beta |\widehat{g-1}|)*m_1'm_1^{-1}|\widehat{T_L T_B B_L\Theta}|\right)\r_s\\
&+\frac{\beta}{C_\beta|k|} \l |Z_1|, m^{-1}p^{-\frac14}m_1'm_1^{-1}|\widehat{T_L T_B B_L\Theta}|\r_s\\
:=&\mathcal{R}_1^1+\mathcal{R}_1^2.
\end{align}
Appealing to Lemma \ref{lem-exchange-freq-weights}, since $\sqrt{m_1'/m_1}=\sqrt{C_\beta}|k|/\sqrt{p}$ and $\langle (k, \eta) \rangle^s \lesssim \langle \eta-\xi \rangle^s \langle (k, \xi) \rangle^s$, then
\begin{equation}
\left(\langle (k,\eta) \rangle^s  m^{-1}p^{-\frac14}\right)(\eta)\sqrt{\frac{m_1'}{m_1}}(\xi )\lesssim\l\eta-\xi\r^{s+\frac32+\delta}\left(\langle (k,\xi) \rangle^s  m^{-1}p^{-\frac14}\right)(\xi)\sqrt{\frac{m_1'}{m_1}}(\eta ).
\end{equation}
Plugging the previous inequality into \eqref{bd:R10} and using Cauchy-Schwarz, 
\begin{align}
\label{bd:R13}
\mathcal{R}_1^1\lesssim& \norm{\sqrt{\frac{m_1'}{m_1}}Z_1}_s\norm{\l\cdot \r^{s+\frac32+\delta}(|\widehat{b}|+\beta |\widehat{g-1}|)*\langle (k, \cdot ) \rangle^s\sqrt{\frac{m_1'}{m_1}}m^{-1}p^{-\frac14}|\widehat{T_L T_B B_L\Theta}|}_{L^2}.
\end{align}
Applying Young's convolution inequality we have 
\begin{equation}
\mathcal{R}_1^1\lesssim \left(\norm{b}_{s+\frac52+\delta}+ \norm{g-1}_{s+\frac52+\delta}\right)\norm{\sqrt{\frac{m_1'}{m_1}}Z_1}_s\norm{\sqrt{\frac{m_1'}{m_1}}m^{-1}p^{-\frac14}T_L T_B B_L \Theta}_s.
\end{equation}
Using \eqref{bd:m1TL} and \eqref{bd:m1Bt}, since $|B_L|\leq1$, see \eqref{eq:BLboundbelow}, we bound the last term as follows
\begin{align}\label{bd:R11}
\norm{\sqrt{\frac{m_1'}{m_1}}m^{-1}p^{-\frac14}T_L T_B B_L\Theta}_s&\leq2 \norm{\sqrt{\frac{m_1'}{m_1}}m^{-1}p^{-\frac14}T_B B_L\Theta}_s
\leq 4\norm{\sqrt{\frac{m_1'}{m_1}}Z_1}_s.
\end{align}
Thus we get
\begin{equation}
\label{bd:R110}
\mathcal{R}_1^1\lesssim \eps \norm{\sqrt{\frac{m_1'}{m_1}}Z_1}^2_s.
\end{equation}
To bound $\mathcal{R}_1^2$, thanks to \eqref{bd:R11} we get 
\begin{equation}
\label{bd:R12}
\mathcal{R}_1^2\leq \frac{4\beta}{C_\beta}\norm{\sqrt{\frac{m_1'}{m_1}}Z_1}^2_s,
\end{equation} 
therefore, by the choice of $C_\beta$, see \eqref{def:Cbeta}, combining \eqref{bd:R110} with \eqref{bd:R12} and
choosing $\eps_0$ small enough, the proof of \eqref{bd:R1est} is over.

\medskip

\noindent\textbf{Estimate on  $\mathcal{R}_2$.}
We prove that
\begin{equation}
\mathcal{R}_2=|k|\sqrt{R}\, \left|\l p^{-\frac{1}{2}}(B_L-1)Z_1, Z_2 \r_s\right| \leq \label{bd:R2est}\frac{1}{32}\left(1-\frac{1}{2\sqrt{R}}\right)\left(\norma{\sqrt{\frac{m'_1}{m_1}} Z_1}{s}^2+\norma{\sqrt{\frac{m'_1}{m_1}} Z_2}{s}^2\right).
\end{equation}
Recall the bound given in \eqref{bd:BL-1}, namely 
\begin{equation}
\label{bd:BL-1sec4}
|B_L-1|\leq \beta(1+\beta)p^{-\frac12},
\end{equation}  
from the definition of $m_1$, see \eqref{def:m1}, we get 
\begin{align}\label{bd:R24}
\mathcal{R}_2\leq \ \frac{\sqrt{R}\beta(1+\beta)}{C_\beta}\norm{\sqrt{\frac{m_1'}{m_1}}Z_1}_s\norm{\sqrt{\frac{m_1'}{m_1}}Z_2}_s.
\end{align}
Hence \eqref{bd:R2est} follows from the definition of $C_\beta$, see \eqref{def:Cbeta}.

\medskip

\noindent\textbf{Estimate on  $\mathcal{R}_3$.}
We aim at proving that
\begin{equation}\label{bd:R3est}
\mathcal{R}_3=|k|\sqrt{R}\, \left| \l m^{-1}p^{-\frac{3}{4}}T_{\eps}T_LB_L\Theta, Z_2 \r_s\right|\leq \frac{1}{32}\left(1-\frac{1}{2\sqrt{R}}\right)\left(\norma{\sqrt{\frac{m'}{m}} Z_1}{s}^2+\norma{\sqrt{\frac{m'}{m}} Z_2}{s}^2\right).
\end{equation}
First of all, notice that the following inequality holds true
\begin{align}
\label{bd:kp12}
|k|p^{-\frac12}(t;k,\eta)&\lesssim (|p'|p^{-1})(t;k,\eta)+|k|^2p^{-1}(t;k,\eta)\lesssim \dfrac{w'}{w}(t; k, \eta)+\dfrac{m_1'}{m_1}(t; k, \eta).
\end{align}
Therefore, we get 
\begin{equation}
\label{bd:R30}
\mathcal{R}_3\lesssim
 \norm{\sqrt{\frac{w'}{w}}m^{-1}p^{-\frac14}T_\eps T_LB_L\Theta}_s \norm{\sqrt{\frac{w'}{w}}Z_2}_s+\norm{\sqrt{\frac{m_1'}{m_1}}m^{-1}p^{-\frac14}T_\eps T_LB_L\Theta}_s \norm{\sqrt{\frac{m_1'}{m_1}}Z_2}_s.
\end{equation}
By applying \eqref{eq:exchangem1}, \eqref{eq:exchangew} followed by \eqref{bd:m1TL}, \eqref{bd:wTL}, since $|B_L|\leq 1$ we infer
\begin{align}
\mathcal{R}_3\lesssim \eps \left(
\norm{\sqrt{\frac{w'}{w}}Z_1}_s \norm{\sqrt{\frac{w'}{w}}Z_2}_s+\norm{\sqrt{\frac{m_1'}{m_1}}Z_1}_s \norm{\sqrt{\frac{w'}{w}}Z_2}_s+\norm{\sqrt{\frac{m_1'}{m_1}}Z_1}_s \norm{\sqrt{\frac{m_1'}{m_1}}Z_2}_s\right).
\end{align}
Finally, by the identity in \eqref{eq:sumofweight} we have
\begin{equation}
\label{bd:R3}
\mathcal{R}_3\lesssim \frac{\eps}{\delta} \left(
\norm{\sqrt{\frac{m'}{m}}Z_1}_s^2 +\norm{\sqrt{\frac{m'}{m}}Z_2}_s^2\right).
\end{equation}
Since $\delta=C_0\eps$, if we choose  $C_0$ sufficiently big (depending only on $\beta,s,R$) we get \eqref{bd:R3est}.

\medskip

\noindent\textbf{Estimate on $\mathcal{R}_4$.}
The goal is again to show that
\begin{equation}\label{bd:R4est}
\mathcal{R}_4=|k|\sqrt{R}\, \left| \l m^{-1}p^{-\frac{3}{4}}T_LB_\eps T_B B_L\Theta, Z_2 \r_s\right|\leq \frac{1}{32}\left(1-\frac{1}{2\sqrt{R}}\right)\left(\norma{\sqrt{\frac{m'}{m}} Z_1}{s}^2+\norma{\sqrt{\frac{m'}{m}} Z_1}{s}^2\right).
\end{equation}
Applying the reasoning that for $\mathcal{R}_3$ gives \eqref{bd:R30}, here we get
\begin{align}
\label{bd:R40}
\mathcal{R}_4\lesssim
\norm{\sqrt{\frac{w'}{w}}m^{-1}p^{-\frac14}T_LB_\eps T_LB_L\Theta}_s \norm{\sqrt{\frac{w'}{w}}Z_2}_s+\norm{\sqrt{\frac{m_1'}{m_1}}m^{-1}p^{-\frac14}T_LB_\eps T_LB_L\Theta}_s \norm{\sqrt{\frac{m_1'}{m_1}}Z_2}_s.
\end{align}
Therefore, by applying first \eqref{bd:wTL}, \eqref{bd:exchangem2} and again \eqref{bd:wTL}, since $|B_L|\leq 1$  and \eqref{eq:sumofweight} holds true, we get 
\begin{align}
\mathcal{R}_4\lesssim &\ \frac{\eps}{\delta} \left(\norm{\sqrt{\frac{m'}{m}}Z_1}_s^2+\norm{\sqrt{\frac{m'}{m}}Z_2}_s^2\right),
\end{align} 
 The proof of \eqref{bd:R4est} is over.

\medskip

\noindent\textbf{Estimate on $\mathcal{R}_5$.}
We show that
\begin{align}
\mathcal{R}_5=&\frac{1}{4\sqrt{R}}\left|\l p' p^{-\frac34} m^{-1} \left((b-\beta g)\Delta_L^{-1} T_L T_B B_L\Theta\right),Z_2\r_s\right|\\
\label{bd:R5est}\leq& \frac{1}{32}\left(1-\frac{1}{2\sqrt{R}}\right)\left(\norma{\sqrt{\frac{m_1'}{m_1}} Z_1}{s}^2+\norma{\sqrt{\frac{m_1'}{m_1}} Z_1}{s}^2\right).
\end{align}
Since $|p|'\leq 2|k|p^{\frac12}$, by using \eqref{def:Deltat-1} we get 
\begin{align}
\mathcal{R}_5\leq& \frac{1}{2\sqrt{R}}\l m^{-1} p^{-\frac14}  \left((|\widehat{b}|+\beta| \widehat{g-1}|)*|k|p^{-1}|\widehat{T_LT_B B_L\Theta}|\right),|Z_2|\r_s\\
&+\frac{\beta}{2\sqrt{R}}\l m^{-1}p^{-\frac14}|k|p^{-1}|\widehat{T_LT_B B_L\Theta}|,|Z_2|\r_s.
\end{align}
Up to the constant $2\sqrt{R}$, the previous bound has the structure of \eqref{bd:form1}-\eqref{bd:form2} with $Z_1$ replaced by $Z_2$. Thus we repeat the computations performed for $\mathcal{R}_1$ to obtain
\begin{equation}
\mathcal{R}_5\leq \frac{1}{C_\beta}\frac{2\beta+C\eps}{\sqrt{R}} \norm{\sqrt{\frac{m_1'}{m_1}}Z_1}_s\norm{\sqrt{\frac{m_1'}{m_1}}Z_2}_s,
\end{equation}
which implies \eqref{bd:R5est} thanks to the choice of $C_\beta$, see \eqref{def:Cbeta}.

\medskip

\noindent\textbf{Estimate on $\mathcal{R}_6$.}
Since $|p'|\leq |k|p^{\frac12}$, analogously to what was done for the term $\mathcal{R}_2$ we get
\begin{equation}
\label{bd:R6est}
\mathcal{R}_6=\frac14\left|\l Z_1, \frac{p'}{p} (B_L-1)Z_1\r_s\right|\leq  \frac{1}{32}\left(1-\frac{1}{2\sqrt{R}}\right)\norma{\sqrt{\frac{m_1'}{m_1}} Z_1}{s}^2.
\end{equation}

\medskip

\noindent\textbf{Estimate on $\mathcal{R}_7$.}
We want to prove that
\begin{equation}
\label{bd:R7est}
\mathcal{R}_7=\frac14\left| \l \frac{p'}{p} Z_1, m^{-1}p^{-\frac14} T_{\eps}T_LB_L\Theta\r_s\right|\leq \frac{1}{32}\left(1-\frac{1}{2\sqrt{R}}\right)\norma{\sqrt{\frac{m_1'}{m_1}} Z_1}{s}^2.
\end{equation}
Recall the definition of $w'/w$, see \eqref{def:w} to have that
\begin{equation}
\mathcal{R}_7\leq\norm{\sqrt{\frac{w'}{w} }m^{-1}p^{-\frac14}T_\eps T_L B_L\Theta}_s\norm{\sqrt{\frac{w'}{w}}Z_1}_s.
\end{equation}
Consequently, by using \eqref{eq:exchangem1}, \eqref{bd:wTL} with the fact that $|B_L|\leq 1$, we get
\begin{align}
\mathcal{R}_7&\lesssim \eps \left(\norm{\sqrt{\frac{w'}{w}}Z_1}_s^2+\norm{\sqrt{\frac{m_1'}{m_1}}Z_1}_s^2\right)
\lesssim \frac{\eps}{\delta}\norm{\sqrt{\frac{m'}{m}}Z_1}_s^2,
\end{align}
where in the last line we have used \eqref{eq:sumofweight}, hence \eqref{bd:R7est} follows upon choosing $C_0$ sufficiently big.
\medskip

\noindent\textbf{Estimate on $\mathcal{R}_8$.}
The last step is to prove that
\begin{equation}
\label{bd:R8est}
\mathcal{R}_8=\frac14\left| \l  \frac{p'}{p} Z_1, m^{-1}p^{-\frac14} T_LB_{\eps}T_BB_L\Theta\r_s\right|\leq \frac{1}{32}\left(1-\frac{1}{2\sqrt{R}}\right)\norma{\sqrt{\frac{m'}{m}} Z_1}{s}^2.
\end{equation}
Using  \eqref{bd:wTL} and \eqref{bd:exchangem2} we obtain that 
\begin{equation}
\mathcal{R}_8\lesssim \eps \left(\norm{\sqrt{\frac{w'}{w}}m^{-1}p^{-\frac14}T_BB_L\Theta}_s +\norm{\sqrt{\frac{m_1'}{m_1}}m^{-1}p^{-\frac14}T_BB_L\Theta}_s\right) \norm{\sqrt{\frac{w'}{w}}Z_1}_s.
\end{equation}
Then, thanks to \eqref{bd:wBt}, \eqref{bd:m1Bt} and $|B_L|\leq 1$, we get
we get
\begin{align}
\mathcal{R}_8&\lesssim \eps \left(\norm{\sqrt{\frac{w'}{w}}Z_1}_s^2+\norm{\sqrt{\frac{m_1'}{m_1}}Z_1}_s^2\right)\lesssim  \frac{\eps}{\delta}\norm{\sqrt{\frac{m'}{m}}Z_1}_s^2,
\end{align}
where in the last line we have used \eqref{eq:sumofweight}.
This proves \eqref{bd:R8est}.

\subsection{Proof of Theorem \ref{thm:Close-main} and Corollary \ref{cor:Omega}}\label{sub:proofs}
Having at hand Proposition \eqref{prop:Esfinal}, we can now prove Theorem \ref{thm:Close-main} and Corollary \ref{cor:Omega}.
\begin{proof}[Proof of Theorem \ref{thm:Close-main}]
	By means of Proposition \ref{prop:Esfinal}, we know that
	\begin{equation}
		\norm{Z_1(t)}_{s}^2+\norm{Z_2(t)}_s^2\lesssim\norm{Z_1(0)}_{s}^2+\norm{Z_2(0)}_s^2. 
	\end{equation}
	Then, from the explicit definition of $w$, see \eqref{expression:w}, we deduce 
	\begin{equation}
		w\lesssim  \jap{t}^\frac12 \jap{k,\eta}.
	\end{equation}
	Hence, recalling the definitions of $Z_1, Z_2$ given in \eqref{eq:Z1Z2again} and since $m=w^{-\delta}m_1^{-1}$, being $m_1$ uniformly bounded, we have that
	\begin{align}
		\norm{p^{-\frac14}\Theta_k(t)}^2_s+\norm{p^{\frac14}Q_k(t)}^2_s=&\ \norm{m_1w^{\delta}Z_1(t)}^2_s+\norm{m_1w^{\delta}Z_2(t)}^2_s\\
		\lesssim&\ \jap{t}^{\delta}\left(\norm{Z_1(t)}^2_{s+\delta}+\norm{Z_2(t)}^2_{s+\delta}\right)\\
		\lesssim&\ \jap{t}^{\delta}\left(\norm{Z_1(0)}^2_{s+\delta}+\norm{Z_2(0)}^2_{s+\delta}\right)
	\end{align}
	where the constant hidden in the last inequalities depends on the uniform bound on $m_1$. Exploiting the definition of $Z_1(0), Z_2(0)$, by the rough bound $\delta \leq \frac12$ the proof of Theorem \ref{thm:Close-main} is over.
\end{proof}

\begin{proof}[Proof of Corollary \ref{cor:Omega}]
	It is enough to prove that
	\begin{align}
		\label{bd:pfCorOm}\norm{p^{-\frac14}\Omega(t)}_s&\lesssim \norm{p^{-\frac14}\Theta(t)}_s,\\
		\label{bd:pfCorTh0}\norm{\Theta(0)}_s&\lesssim\norm{\Omega(0)}_s.
	\end{align}
	Since $\Omega=B_t\Theta=T_BB_L \Theta$, thanks to \eqref{def:TB} and $|B_L|\leq1$ we have 
	\begin{equation}
		\label{bd:pfCor0}
		\norm{p^{-\frac14}\Omega(t)}_s\leq \norm{p^{-\frac14}\Theta(t)}_s+\norm{p^{-\frac14}B_\eps T_BB_L\Omega(t)}_s.
	\end{equation}
	To treat the second term, we  argue as in  \eqref{bd:exchangem2} and \eqref{bd:m1Bt} and commute $p^{-\frac14}$ with $B_\eps T_B$.
	We therefore obtain
	\begin{align}
		\norm{p^{-\frac14}B_\eps T_BB_L\Omega(t)}_s\lesssim& \ \eps \norm{p^{-\frac14}\Omega(t)}_s .
	\end{align}
	As $\eps_0$ is small enough, we can absorb the last term in the left-hand side of \eqref{bd:pfCor0}, hence obtaining \eqref{bd:pfCorOm}. The proof of \eqref{bd:pfCorTh0}  simply comes from the fact that $\Theta(0)=(I-\beta \dy\Delta^{-1})\Omega(0)$. 
\end{proof}

\subsection*{Acknowledgements}
We would like to thank Professor T. Gallay for invaluable comments that improved the manuscript.
The work of MCZ was partially supported by the Royal Society through a 
University Research Fellowship (URF\textbackslash R1\textbackslash 191492). This project has received funding from the European Research Council (ERC) under the European Union's Horizon 2020 research and innovation program Grant agreement No 637653, project BLOC ``Mathematical Study of Boundary Layers in Oceanic Motion’’. RB was partially supported by the SingFlows project, grant ANR-18-CE40-0027 of the French National Research Agency (ANR) and from the GNAMPA group of INdAM (GNAMPA project 2019, INdAM, Italy). MD was partially supported from the GNAMPA group of
INdAM (GNAMPA project 2019, INdAM, Italy).

\bibliographystyle{abbrv}
\bibliography{biblio}

\end{document}